\documentclass[11pt, a4paper]{amsart}
\usepackage{graphicx} % Required for inserting images
\usepackage[margin=1in]{geometry}
\title{Frobenius--Witt cotangent complexes}
\author{Kanau Shimada}
\date{}

\address{Department of Mathematics, Institute of Science Tokyo, 2-12-1 Ookayama, Meguro, Tokyo 152-8551}
\email{shimada.k.aj@m.titech.ac.jp}

\usepackage{amsmath}
\usepackage{amssymb}
\usepackage{latexsym}
\usepackage{amsthm}
\usepackage{amsfonts}
\usepackage{mathtools}
\usepackage{stmaryrd}
\usepackage{tikz}
\usepackage{tikz-cd}
\usepackage{slashed}

\theoremstyle{plain}
\newtheorem{theorem}{Theorem}[section]
\newtheorem{lemma}[theorem]{Lemma}
\newtheorem{prop}[theorem]{Proposition}
\newtheorem{coro}[theorem]{Corollary}
\newtheorem{conj}[theorem]{Conjecture}

\theoremstyle{definition}
\newtheorem{definition}[theorem]{Definition}
\newtheorem{remark}[theorem]{Remark}
\newtheorem{example}[theorem]{Example}

\usepackage{mathrsfs}
\usepackage{blindtext}
\usepackage{hyperref}


\begin{document}

\maketitle
\begin{abstract}
    We introduce the notion of the Frobenius--Witt cotangent complex, which can be considered as a derived variant of the module of Frobenius--Witt differentials defined by T. Saito. This new object also can be seen as an arithmetic variant of the notion of cotangent complex. Furthermore, we establish a relationship between Frobenius--Witt cotangent complexes and the regularity of noetherian local rings, which can be considered as a derived variant of Saito's regularity criterion. This proof relies heavily on computations of Frobenius--Witt cotangent complexes in the case of perfectoid rings. We also study the deformation theory of $\delta$-structures using Frobenius--Witt cotangent complexes.
\end{abstract}

\tableofcontents

\section{Introduction}

M. Andr\'{e} and D. Quillen introduced the cotangent complex around 1970. It is a derived version of the module of K\"{a}hler differentials and is a fundamental invariant controlling the deformation theory of schemes. More recently, in \cite{BMS19}, the following theorem was proved.
\begin{theorem}[\cite{BMS19}]
    The functor $\operatorname{Ring}\to \mathcal{D}(\mathbb Z):A\mapsto \mathbb L_{A/\mathbb Z}$ is an fpqc sheaf.
\end{theorem}
This property is of great importance, since, through the Hodge--Tate comparison theorem, it leads to the sheafiness of prismatic cohomology.

On the other hand, in 2022, T. Saito introduced in \cite{Sai22} the Frobenius--Witt differentials $F\Omega_A$ as an arithmetic analogue of the module of Kähler differentials. In that paper, he proved a criterion for the regularity of mixed-characteristic local rings in terms of Frobenius--Witt differentials. A basic ingredient underlying this theory is the following fundamental exact sequence.
\begin{lemma}[\cite{Sai22}]
    Let $A\to B$ be a map of $\mathbb Z_{(p)}$-algebras. Then there is the following right exact sequence of $B/p$-modules.
    \[F\Omega_A\otimes _AB\to F\Omega_B\to F^*\Omega_{(B/p)/(A/p)}\to0,\] where $F^*$ is the pullback along the Frobenius map on $B/p$. 
\end{lemma}
However, just as in the case of Kähler differentials, this exact sequence is not left exact without additional assumptions such as smoothness. In this paper, we define the Frobenius--Witt cotangent complex $F\mathbb L_A$ for a ring $A$, more precisely for an animated $\mathbb Z_{(p)}$-algebra, as a derived version of the Frobenius--Witt differentials. We then extend the above right exact sequence to a long exact sequence. More precisely, we prove the following fundamental theorem.
\begin{theorem}[Theorem \ref{thm5.5}]
    Let $A\to B$ be a map of animated $\mathbb Z_{(p)}$-algebras. Then there is a canonical fiber sequence
    \[F\mathbb L_A\otimes_A^LB\to F\mathbb L_B\to F^*(\mathbb L_{B/A}\otimes _B^LB/^Lp)\to \] in the derived category $\mathcal D(B/^Lp)$, where $F$ is the derived pullback along the Frobenius map on $B/^Lp$.
\end{theorem}
This is analogous to the existence of the fiber sequence for the cotangent complex
\[\mathbb L_{B/A}\otimes^L_BC\to \mathbb L_{C/A}\to \mathbb L_{C/B}\to \] for a sequence of rings $A \to B \to C$.
Furthermore, we prove that the Frobenius--Witt cotangent complex satisfies the following basic properties, parallel to those of the cotangent complex.
\begin{itemize}
    \item the functor $A\mapsto F\mathbb L_A$ satisfies etale localization,
    \item the functor $A\mapsto F\mathbb L_A$ commutes with all sifted colimits and pushouts
\end{itemize}
The analogue of the theorem mentioned at the beginning also holds for the Frobenius--Witt cotangent complex.
\begin{theorem}[Theorem \ref{thm5.11}]
    The functor $\operatorname{Ring}_{\mathbb Z_{(p)}/}\to \mathcal D(\mathbb F_p):A\mapsto F\mathbb L_A$ is an fpqc sheaf.
\end{theorem}
This statement does not hold for Frobenius--Witt differentials themselves, and thus provides another advantage of considering the Frobenius--Witt cotangent complex. On the other hand, the Frobenius--Witt cotangent complex has the expected properties as a derived version of the Frobenius--Witt differentials.
\begin{itemize}
    \item The $0$-th homology group of $F\mathbb L_A$ coincides with $F\Omega_A$,
    \item if $A$ is a smooth $\mathbb Z_{(p)}$-algebra, $F\mathbb L_A$ coincides with $F\Omega_A$.
\end{itemize}
These results show that the Frobenius--Witt cotangent complex plays for Saito's Frobenius--Witt differentials the same role as the ordinary cotangent complex plays for Kahler differentials.

As an application of the Frobenius--Witt cotangent complex, we prove the following derived extension of Saito’s regularity criterion.
 \begin{theorem}[Corollary \ref{cor6.12}]
     Let $R$ be a $p$-torsionfree noetherian local ring of residue characteristic $p$. Assume that the ring $R/p$ is $F$-finite. Then the following conditions are equivalent:\begin{enumerate}
        \item $R$ is regular,
        \item $F\mathbb L_R$ is a free $R/p$-module of rank $\operatorname{dim}R+\operatorname{log}_p[k:k^p]$,
    \end{enumerate}
    where $k$ is the residue field of $R$.
\end{theorem}

This theorem can be proved conceptually by using the following vanishing theorem.
\begin{theorem}[Theorem \ref{thm6.7}]
    Let $R$ be a perfectoid ring. Then $F\mathbb L_R\simeq 0$ holds.
\end{theorem}
Such a connection between the theory of Frobenius--Witt differentials and the theory of perfectoid rings is a completely new phenomenon, which had not appeared in previous work. This theorem may be regarded as a mixed-characteristic analogue of the vanishing of the cotangent complex over $\mathbb{F}_p$ for perfect rings.

As another application of the Frobenius--Witt cotangent complex, we study its relation to $\delta$-structures on rings. A $\delta$-structure is an endomorphism $\delta$ of the underlying set of a ring satisfying certain conditions, and the map $\phi := (-)^p + p\delta$ is then a lift of Frobenius. In this paper, given a Frobenius lift $\phi$, we describe the obstruction class to the existence of a $\delta$-structure $\delta$ satisfying the above formula, as well as the set of all such $\delta$-structures. More precisely, we prove the following theorem.
\begin{theorem}[Theorem \ref{D.2}]
    Let $A$ be a $\mathbb Z_{(p)}$-algebra equipped with a Frobenius lift $\phi$.
\begin{enumerate}
\item Suppose that $A$ admits a $\delta$-structure $\delta_0$ satisfying $\phi = (-)^p + p\delta_0$.
Then the set
\[\{\delta\in \operatorname{End}_{\operatorname{Set}}(A)|\,\delta \text{ is a }\delta\text{-structure on }A \text{ satisfying} \ \phi=(-)^p+p\delta\}\]
is a torsor under $\operatorname{Ext}^0_{A/^Lp}((F\mathbb L_A,w(p)),(A[p],0))$.

\item There exists an element
\[\xi\in \operatorname{Ext}^1_{A/^Lp}((F\mathbb L_A,w(p)),(A[p],0))\]
such that $\xi$ is zero if and only if $A$ admits a $\delta$-structure $\delta$ satisfying $\phi=(-)^p+p\delta$. Furthermore, the element $\xi$ only depends on $A$ and $\phi$.
\end{enumerate}
\end{theorem}
This expresses the deformation-theoretic aspect of the Frobenius--Witt cotangent complex.

\vspace{2mm}
\textbf{Structure of this paper}. 
\begin{itemize}
    \item In Section 2, we reformulate Saito’s Frobenius--Witt differentials in terms of the arithmetic square-zero extension $W_2(A,M,m)$ of a ring $A$.
    \item In Section 3, we review the theory of animation, following \cite{CS24}. We also prove several properties of animation that will be needed in order to introduce the Frobenius--Witt cotangent complex.
    \item In Section 4, using the results prepared in Section 3, we introduce the animation of $W_2(A,M,m)$. We then define a derived version of Frobenius--Witt derivations, and introduce the Frobenius--Witt cotangent complex as the object corepresenting the space of such derived Frobenius--Witt derivations. We subsequently discuss several essentially equivalent definitions of this complex.
    \item In Section 5, we study the basic properties of the Frobenius--Witt cotangent complex.
    \item In Section 6, we compute the Frobenius--Witt cotangent complex for several classes of rings, including perfectoid rings and quasiregular semiperfectoid rings. As an application of these computations, we prove the regularity criterion.
    \item In Section 7, we study the relation between the Frobenius--Witt cotangent complex and the deformation theory of $\delta$-structures.
    \item In Section 8, assuming the existence of $\mathbb F_1$, we formulate a conjectural lift of the Frobenius--Witt cotangent complex to $\mathbb Z_p$.
\end{itemize}

\vspace{2mm}
\textbf{Notations}. In this paper, a ring is always assumed to be a unital commutative ring. We fix a prime number $p$. We denote by $\operatorname{Ring}_A$ the category of $A$-algebras for each ring $A$. We use the term $\infty$-category in the sense of quasi-categories as in \cite{Lur09}, and the term animation in the sense of \cite[Section 5.1]{CS24}. We denote by $\mathcal S$ the $\infty$-category of spaces. For an $\infty$-category $\mathscr C$ and its objects $X, Y\in \mathscr C$, we denote by $\operatorname{Hom}_{\mathscr C}(X,Y)$ the mapping space from $X$ to $Y$. We denote by $\operatorname{Fun}(\mathscr C,\mathscr D)$ the $\infty$-category of functors from an $\infty$-category $\mathscr C$ to an $\infty$-category $\mathscr D$. For a map of (animated) rings $A\to B$, we denote by $\mathbb L_{B/A}$ its cotangent complex in the sense of \cite[Definition 2.33]{Mao24}.

\vspace{2mm}
\textbf{Acknowledgments}. 
The author is deeply grateful to Ryo Ishizuka for his contribution to Lemma \ref{lem6.6} and Lemma \ref{lem6.10.1}. The author is also deeply grateful to Ryoma Takeuchi for valuable discussions and comments. The author was inspired by his talk to remove the completeness assumption in Theorem \ref{thm6.16}. The author is also deeply grateful to Kazuki Hayashi for reading the draft version of this paper and giving many valuable comments. Finally, the author would like to express his sincere gratitude to his supervisor Yuichiro Taguchi for constant encouragement and guidance.

\section{Frobenius--Witt derivations via arithmetic extensions}
In this section, we recall the notion of Frobenius--Witt derivations introduced in \cite{Sai22}. After that, we introduce a new object $W_2(A,M,m)$ and reformulate the Frobenius--Witt derivations in terms of this.

\begin{definition}\label{def2.1}
    We write $P(X,Y):=\frac{(X+Y)^p-X^p-Y^p}{p}\in \mathbb Z[X,Y]$.
\end{definition}
\begin{definition}\label{def2.2}
    Let $A$ be a ring and $M$ be an $A$-module. A map $w:A\to M$ is called a {\em Frobenius--Witt derivation} if the following conditions hold for every $x,y\in A$:\begin{enumerate}
        \item $w(x+y)=w(x)+w(y)-P(x,y)w(p)$,
        \item $w(xy)=x^pw(y)+y^pw(x)$.
    \end{enumerate}
    For each element $m\in M$, we denote by $\operatorname{FWDer}(A,(M,m))$ the set of Frobenius--Witt derivations whose value at $p\in A$ is $m$. 
\end{definition}

\begin{remark}\label{rem2.3}
    \begin{itemize}
        \item As in \cite{Sai22}, we mainly work with $\mathbb Z_{(p)}$-algebras.
        \item Let $A$ be a $\mathbb Z_{(p)}$-algebra, $M$ be an $A$-module and $w:A\to M$ be a Frobenius--Witt derivation. Then $w$ has its values in the $p$-torsion part $M[p]$ of $M$ (\cite[Lemma 1.2]{Sai22}). Thus, we mainly consider the case that $M$ is an $A/p$-module.
    \end{itemize}
\end{remark}

To rewrite Definition \ref{def2.2}, we introduce the following new object:
\begin{definition}\label{def2.4}
    Let $A$ be a $\mathbb Z_{(p)}$-algebra, $M$ be an $A/p$-module and $m\in M$ be an element. The arithmetic extension $W_2(A,M,m)$ of $A$ by $(M,m)$ is the ring defined by the following formula:\begin{enumerate}
        \item the underlying set is $W_2(A,M,m):=A\times M$,
        \item additions are defined by $(a,x)+(b,y):=(a+b,x+y-P(a,b)m)$ for every $(a,x), (b,y)\in A\times M$,
        \item multiplications are defined by $(a,x)(b,y):=(ab,a^py+b^px)$ for every $(a,x), (b,y)\in A\times M$.
    \end{enumerate}
\end{definition}

\begin{remark}
    Definition \ref{def2.4} is well-defined. Let $(a,x), (b,y), (c,z)$ be elements of $A\times M$. The assertion follows from the following computations:
    \begin{itemize}
        \item (associativity of addition). Since $((a,x)+(b,y))+(c,z)=(a+b+c,x+y-P(a,b)m+z-P(a+b,c)m)$ holds for each $(a,x), (b,y), (c,z)\in A\times M$, it suffices to show the following equality:\[P(X,Y)-P(X+Y,Z)=P(Y,Z)-P(X,Y+Z)\in \mathbb Z[X,Y,Z].\] This follows from the fact that the both sides coincide with $\frac{-(X+Y+Z)^p+X^p+Y^p+Z^p}{p}$.
        \item (associativity of multiplication). This follows from the following computation. \[((a,x)(b,y))(c,z)=(abc,a^pb^pz+b^pc^px+c^pa^py).\]
        \item The zero element is $(0,0)$ and the identity element is $(1,0)$.
        \item (the existence of the inverse). For any element $(a,x)$, its additive inverse is given by $(-a,-x+P(a,-a)m$.
        \item (distributive law). Since
        \[(a,x)((b,y)+(c,z))=(a(b+c),a^p(y+z-P(b,c)m)+(b+c)^px),\]
        \[(a,x)(b,y)+(a,x)(c,z)=(ab+ac,(a^py+b^px)+(a^pz+c^px)-P(ab,ac)m)\] hold, it suffices to show the following equations:\[((b+c)^p-b^p-c^p)x=0\in M,\] \[P(ab,ac)=a^pP(b,c)\in A.\] The former follows from the fact that $M$ is $p$-torsion, and the latter follows from a direct computation.
    \end{itemize}
\end{remark}\label{lem2.5}

\begin{remark}\label{rem2.6}
    We denote by $\mathrm{pr}_1:W_2(A,M,m)\to A$ the projection onto the first component and by $\mathrm{pr}_2:W_2(A,M,m)\to M$ the projection onto the second component. Note that the former one is a ring homomorphism, but the latter is not.
\end{remark}
\begin{example}\label{exa2.7}
    Let $A$ be a $\mathbb Z_{(p)}$-algebra, $M$ be an $A/p$-module and $m\in M$ be an element. Then we have $p=(p,m)\in W_2(A,M,m)$. To see this, it suffices to prove the following equality:\[-\sum_{n=1}^{p-1}P(n,1)m=m\in M.\] Since $M$ is $p$-torsion, it suffices to show the following equality:\[\sum_{n=1}^{p-1}P(n,1)=p^{p-1}-1\in \mathbb Z.\] This follows from the following elementary computation:\begin{align*}
        \sum_{n=1}^{p-1}P(n,1)&=\frac{1}{p}\sum_{n=1}^{p-1}((n+1)^p-n^p-1)\\&=\frac{1}{p}((\sum_{n=1}^{p-1}(n+1)^p)-(\sum_{n=1}^{p-1}n^p)-(p-1))\\&=\frac{1}{p}((\sum_{n=2}^{p}n^p)-(\sum_{n=1}^{p-1}n^p)-(p-1))\\&=p^{p-1}-1.
    \end{align*}
    This observation implies that the datum $m\in M$ can be recovered from the data consisting of the ring structure on $W_2(A,M,m)=A\times M$ and the map $\mathrm{pr}_2:W_2(A,M,m)=A\times M\to M$ of sets.
\end{example}
\begin{prop}\label{pro2.8}
    Let $A$ be a $\mathbb Z_{(p)}$-algebra and $M$ be an $A/p$-module. Then the following sets are canonically isomorphic.\begin{enumerate}
        \item the set of Frobenius--Witt derivations $w:A\to M$,
        \item the set of ring homomorphisms $s:A\to W_2(A,M,m)$ such that the composition $\mathrm{pr}_1\circ s:A\to A$ is the identity map for some $m\in M$.
    \end{enumerate}
\end{prop}
\begin{proof}
    Let $w:A\to M$ be a Frobenius--Witt derivation. Then the corresponding element in (2) is the map defined by $s:A\to W_2(A,M,w(p));a\mapsto(a,w(a))$. Then the following computations ensure that this map is a ring homomorphism:\begin{itemize}
        \item $s(a)s(b)=(a,w(a))(b,w(b))=(ab,a^pw(b)+b^pw(a))=(ab,w(ab))=s(ab)$,
        \item $s(a+b)=(a+b,w(a+b))=(a+b,w(a)+w(b)-P(a,b)w(p))=(a,w(a))+(b,w(b))$,
        \item $s(0)=(0,w(0))=0$,
        \item $s(1)=(1,w(1))=(1,0)=1,$ where we used the fact $w(1)=0$ (\cite[Lemma 1.2]{Sai22}).
    \end{itemize}

    Conversely, let $s:A\to W_2(A,M,m)$ be a map as in (2). Then the corresponding map of (1) is defined by the composition $w:=\mathrm{pr}_2\circ s:A\to M$. One can prove that this map is indeed a Frobenius--Witt derivation by the following computations:\begin{itemize}
        \item (multiplication).\begin{align*}
            w(ab)&=\mathrm{pr}_2(s(ab))\\&=\mathrm{pr}_2(s(a)s(b))\\&=\mathrm{pr}_2((a,w(a))(b,w(b)))\\&=a^pw(b)+b^pw(a),
        \end{align*}
        \item (addition).\begin{align*}
             w(a+b)&=\mathrm{pr}_2(s(a+b))\\&=\mathrm{pr}_2((a,w(a))+(b,w(b)))\\&=\mathrm{pr}_2((a+b,w(a)+w(b)-P(a,b)m))\\&=w(a)+w(b)-P(a,b)m\\&=w(a)+w(b)-P(a,b)w(p),
        \end{align*} where the last equation follows from $m=w(p)\in M$, as shown in the above example.
        
    \end{itemize}
    The fact that these correspondences are one-to-one follows immediately from these constructions.
\end{proof}

In the end of this subsection, we observe the functoriality of $W_2(A,M,m)$.

\begin{definition}\label{def2.9}
\begin{itemize}
    \item We denote by $\operatorname{Ring}$ the category of rings.
    \item We denote by $\operatorname{RingMod}$ the category of pairs $(A,M)$, where $A$ is a ring and $M$ is an $A$-module. Its morphisms are of the form $(\varphi,\alpha):(A,M)\to (B,N)$, where $\varphi:A\to B$ is a ring homomorphism and $\alpha:M\to N$ is a $\varphi$-semilinear map. 
    \item We denote by $\operatorname{RingMod^{p-tor}}$ the category of pairs $(A,M)$, where $A$ is a ring and $M$ is an $A/p$-module. Its morphisms are as in the preceding category.
    \item We denote by $\operatorname{AlgMod^{p-tor}_*}$ the category of triples $(A,M,m)$, where $A$ is a $\mathbb Z_{(p)}$-algebra, $M$ is an $A/p$-module and $m$ is an element of $M$. Its morphisms are of the form $(\varphi,\alpha):(A,M,m)\to (B,N,n)$, where $\varphi:A\to B$ is a ring homomorphism and $\alpha:M\to N$ is a $\varphi$-semilinear map such that $\alpha(m)=n$.
\end{itemize}
\end{definition}
\begin{remark}\label{rem2.10}
    \begin{enumerate}
        \item The category $\operatorname{RingMod^{p-tor}}$ is a full subcategory of $\operatorname{RingMod}$. 
        \item There exists an equivalence of categories $\operatorname{AlgMod^{p-tor}_*}\simeq (\operatorname{RingMod^{p-tor}})_{(\mathbb Z_{(p)},\mathbb F_p)/}$, where the right hand side is the coslice category.
    \end{enumerate}
\end{remark}
\begin{definition}\label{def2.11}
    We denote by $W_2(-,-,-):\operatorname{AlgMod^{p-tor}_*}\to \operatorname{Ring}$ the functor defined by $(A,M,m)\mapsto W_2(A,M,m)$. For a map $(\varphi,\alpha):(A,M,m)\to (B,N,n)$, the induced map $W_2(A,M,m)\to W_2(B,N,n)$ is given by $(a,x)\mapsto (\varphi(a),\alpha(x))$. (One can directly prove that this map is a ring homomorphism using the $\varphi$-semilinearity of $\alpha$).
\end{definition}

\section{Preliminaries on category theory}
In this subsection, we recall some basic facts about the notion of animation. We also prove some preliminary results in order to introduce an animated version of Frobenius--Witt derivations.

\begin{definition}\label{def3.1}
    Let $\mathscr C$ be a category. An object $X\in \mathscr C$ is called {\em compact projective} if and only if the functor $\operatorname{Hom}_{\mathscr C}(X,-):\mathscr C\to \operatorname{Set}$ commutes with all small filtered colimits and all reflexive coequalizers.
    
\end{definition}
\begin{definition}\label{def3.2}
    Let $\mathscr C$ be a category. A set $S$ consisting of objects of $\mathscr C$ is called a {\em set of compact projective generators} of $\mathscr C$ if it consists of compact projective objects of $\mathscr C$ and generates $\mathscr C$ under colimits. 
\end{definition}

\begin{definition}\label{def3.3}
    Let $\mathscr C$ be a category, $S$ be a set of compact projective generators of $\mathscr C$ and $\mathscr C_0$ be the full subcategory of $\mathscr C$ generated by $S$ under all finite colimits. Then we define the {\em animation} $\operatorname{Ani}(\mathscr C)$ of $\mathscr C$ by the $\infty$-category:\[\operatorname{Ani}(\mathscr C):=\operatorname{Fun}_{\prod}(\mathscr C_0^{\mathrm{op}},\mathcal S),\] where the right-hand side is the full subcategory of the $\infty$-category $\operatorname{Fun}(\mathscr C_0^{\mathrm{op}},\mathcal S)$ consisting of functors preserving all finite products.
\end{definition}
\begin{remark}\label{rem3.4}
Let $\mathscr C$ be as in Definition \ref{def3.3}. 
\begin{itemize}
    \item The $\infty$-category $\operatorname{Ani}(\mathscr C)$ depends only on the category $\mathscr C$, not on the choice of $S$. (\cite[Remark A.20]{Mao24}).
    \item The $\infty$-category $\operatorname{Ani}(\mathscr C)$ admits all sifted colimits. (\cite[Proposition 5.5.8.10]{Lur09}).
    \item  Let $\pi_0:\operatorname{Ani}(\mathscr C)\to \operatorname{Ani}(\mathscr C)$ be the functor of $\infty$-categories induced by the truncation functor $\pi_0:\mathcal S\to \mathcal S$ of spaces. Then its essential image is canonically equivalent to the category $\mathscr C$. In this way, we can regard the category $\mathscr C$ as a full subcategory of $\operatorname{Ani}(\mathscr C)$. (\cite[Remark 5.5.8.26]{Lur09})
    \item Each object of $\mathscr C_0$ is compact projective in  $\operatorname{Ani}(\mathscr C)$ and they generate $\operatorname{Ani}(\mathscr C)$ under sifted colimits. (\cite[Proposition 5.5.8.22]{Lur09}).
\end{itemize}
    
\end{remark}

\begin{example}\label{exa3.5}
It is straightforward to check the following assertions. (\cite[Example 5.1.3]{CS24}).
    \begin{itemize}
        \item The category $\operatorname{Set}$ of sets has a set of compact projective generators consisting of all finite sets. The full category $\operatorname{FinSet}$ spanned by this is closed under finite coproducts. Therefore, we can define the animation $\operatorname{Ani(Set)}$ of $\operatorname{Set}$. It is known that this $\infty$-category is canonically equivalent to the $\infty$-category $\mathcal S$ of spaces.
        \item The category $\operatorname{Ring}$ of rings has a set of compact projective generators consisting of all finitely generated polynomial rings. The full subcategory $\operatorname{Poly}$ spanned by this is closed under finite coproducts. Therefore, one can define the animation $\operatorname{Ani}(\operatorname{Ring})$ of the category $\operatorname{Ring}$, whose objects are called animated rings.
        \item The category $\operatorname{Ab}$ of abelian groups has a set of compact projective generators consisting of free abelian groups of finite rank. The full subcategory spanned by this is closed under finite coproducts. Therefore, we can define the animation $\operatorname{Ani}(\operatorname{Ab})$ of $\operatorname{Ab}$, whose objects are called animated abelian groups. It is known that there exists a canonical equivalence $\operatorname{Ani}(\operatorname{Ab})\simeq \mathcal D_{\ge 0}(\mathbb Z)$, where the right-hand side is the full subcategory of the derived $\infty$-category $\mathcal D(\mathbb Z)$ consisting of all connective objects.
    \end{itemize}
\end{example}
To see some examples of animations needed later, we make some 1-categorical preliminaries.

\begin{lemma}\label{lem3.6}
    The category $\operatorname{RingMod}$ defined in Definition \ref{def2.9} has a compact projective generators consisting of pairs $(P,F)$ where $P$ is a finitely generated polynomial ring and $F$ is a finite free $P$-module. The full subcategory spanned by this compact projective generators is closed under finite coproducts.
\end{lemma}
\begin{proof}(This is a well-known fact, but we provide a proof for the reader's convenience).
    First, it is straightforward to check that colimits in the category $\operatorname{RingMod}$ can be computed as follows:\[\operatorname{colim}_{i\in I}(A_i,M_i)\cong (\operatorname{colim}_{i\in I}A_i,\operatorname{colim}_{i\in I}((\operatorname{colim}_{i\in I}A_i)\otimes _{A_i}M_i)),\] where, on the right-hand side, the colimit in the first component is taken in the category of rings and the colimit in the second component is taken in the category of $\operatorname{colim}_{i\in I}A_i$-modules. 

    The fact that any object $(P,F)$ as in the statement is compact projective follows from the following natural bijection \[\operatorname{Hom}_{\operatorname{RingMod}}((P,F),(A,M))\cong A^{\times d}\times M^{\times r},\] where $d$ is the number of variables of $P$, $r$ is the rank of $F$ as a $P$-module and $(A,M)$ is an arbitrary object of $\operatorname{RingMod}$.

    Finally, we prove that objects of the form $(P,M)$ as in the statement generate $\operatorname{RingMod}$ under colimits. Denote by $\mathscr C$ the full subcategory spanned by these objects. Let $(A,M)$ be an arbitrary object of $\operatorname{RingMod}$. There exists a diagram of finitely generated polynomials $\{P_i\}_{i\in I}$ such that $A\cong \operatorname{colim}_{i\in I}P_i$. Then we have $(A,0)\cong \operatorname{colim}_{i\in I}(P_i,0)$, so $(A,0)\in \mathscr C$ holds. Since $(A,0)\coprod(\coprod_{K}(\mathbb Z,\mathbb Z))\cong (A,A^{\oplus K})$, we have $(A,A^{\oplus K})\in \mathscr C$ for any set $K$. Writing $M\cong \operatorname{colim}_{j\in J}A^{\oplus K_j}$, we get $(A,M)\in \mathscr C$.
    
\end{proof}

\begin{lemma}\label{lem3.7}
     The category $\operatorname{RingMod^{p-tor}}$ defined in Definition \ref{def2.9} has a compact projective generators consisting of pairs $(P,F)$ where $P$ is a finitely generated polynomial ring and $F$ is a finite free $P/p$-module. The full subcategory spanned by this compact projective generators is closed under finite coproducts.
\end{lemma}
\begin{proof}
    It follows from the description of colimits in the category $\operatorname{RingMod}$ (see the first paragraph of the proof of Lemma \ref{lem3.6}) that the full subcategory $\operatorname{RingMod^{p-tor}}\subseteq \operatorname{RingMod}$ is closed under colimits. On the other hand, there exists a natural bijection \[\operatorname{Hom}_{\operatorname{RingMod}}((P,F),(A,M))\cong \operatorname{Hom}_{\operatorname{RingMod^{p-tor}}}((P,F/p),(A,M)),\] where $(P,F)$ is a pair of finitely generated polynomial ring $P$ and a finite free $P$-module $F$, and $(A,M)$ is an object of $\operatorname{RingMod^{p-tor}}$. These observations and the Lemma \ref{lem3.6} imply that all objects of the form $(P,F)$ in $\operatorname{RingMod^{p-tor}}$, where $P$ is a finitely generated polynomial ring and $F$ is a finite free $P/p$-module, are compact projective.

    Next, we prove that these objects generate $\operatorname{RingMod^{p-tor}}$ under colimits. Let $(A,M)$ be an arbitrary object of $\operatorname{RingMod^{p-tor}}$. By the above lemma, there exists a small diagram $\{(P_i,F_i)\}_{i\in I}$ of pairs consisting of finitely generated polynomial rings and finite free modules over them such that \[\operatorname{colim}_{i\in I}(P_i,F_i)\cong (A,M).\] This implies that there exists a following isomorphisms:\[\operatorname{colim}_{i\in I}P_i\cong A,\]\[\operatorname{colim}_{i\in I}(F_i\otimes_{P_i}A)\cong M.\] Since $M$ is $p$-torsion, the latter implies $\operatorname{colim}_{i\in I}(F_i/p\otimes_{P_i}A)\cong M$. This implies $\operatorname{colim}_{i\in I}(P_i,F_i/p)\cong (A,M)$. 

    The last statement is immediate from the description of coproducts in $\operatorname{RingMod^{p-tor}}$.
\end{proof}
\begin{lemma}\label{lem3.8}
    Let $\mathscr C$ be a category, $S$ be a set of compact projective generator of $\mathscr C$ and $X\in \mathscr C$ be an object. Then the set $\{X\to X\coprod Y|\,Y\in S\}$ is a compact projective generator of the coslice category $\mathscr C_{X/}$. Moreover, there exists a canonical equivalence of $\infty$-categories $\operatorname{Ani}(\mathscr C_{X/})\simeq \operatorname{Ani}(\mathscr C)_{X/}$.
\end{lemma}
\begin{proof}
    This is \cite[Lemma 2.12]{Mao24} and \cite[Corollary 2.14]{Mao24}.
\end{proof}
Combining Lemma \ref{lem3.7} and Lemma \ref{lem3.8}, we can define the animation $\operatorname{Ani(AlgMod_*^{p-tor})}$ of the category $\operatorname{AlgMod_*^{p-tor}}$.

Next, we recall some properties of functors induced on the $\infty$-category obtained by animation.

\begin{lemma}\label{lem3.9}
    Let $\mathscr C_0$ be a small category admitting all finite coproducts and $\mathscr D$ be an $\infty$-category admitting all sifted colimits. Then the Yoneda embedding functor $j:\mathscr C_0\hookrightarrow \mathscr P _{\Sigma}(\mathscr C_0):=\operatorname{Fun}_{\prod}(\mathscr C_0,\mathcal S)$ induces an equivalence of $\infty$-categories\[\operatorname{Fun}_{\Sigma}(\mathscr P _{\Sigma}(\mathscr C_0),\mathscr D)\simeq \operatorname{Fun}(\mathscr C_0,\mathscr D),\] where the left hand-side is the full subcategory of the $\infty$-category of functors $\operatorname{Fun}(P _{\Sigma}(\mathscr C_0),\mathscr D)$ consisting of all functors preserving all sifted colimits.
    Furthermore, the quasi-inverse of this equivalence is given by the left Kan extension along $j$.
\end{lemma}
\begin{proof}
    This is \cite[Proposition 5.5.8.15]{Lur09}.
\end{proof}

\begin{definition}\label{def3.10}
    Let $\mathscr C$ be a set of category having a set of compact projective generators $S$, $\mathscr C_0$ be the full subcategory of $\mathscr C$ spanned by $S$ and $\mathscr D$ be an $\infty$-category admitting all sifted colimits. Assume that the inclusion $\mathscr C_0\hookrightarrow \mathscr C$ preserves all finite coproducts. Let $F:\mathscr C\to \mathscr D$ be a functor of $\infty$-categories. Then the {\em animation} $\operatorname{Ani}(F)$  of the functor $F$ is the left Kan extension of the composition $\mathscr C_0\hookrightarrow \mathscr C\xrightarrow[]{F}\mathscr D$ along the Yoneda embedding $j:\mathscr C_0\hookrightarrow \operatorname{Ani}(\mathscr C)$.
\end{definition}
\begin{definition}\label{def3.11}
    Let $F:\mathscr C\to \mathscr D$ be a functor of categories such that $\mathscr C$ and $\mathscr D$ have their animations. Then the animation $\operatorname{Ani}(\mathscr C)\to \operatorname{Ani}(\mathscr D)$ of the composition $\mathscr C\xrightarrow[]{F}\mathscr D\hookrightarrow \operatorname{Ani}(\mathscr D)$ is also called the {\em animation} of $F$ and denoted by $\operatorname{Ani}(F)$.
\end{definition}

\begin{remark}\label{rem3.12} 
Let $\mathscr C$, $\mathscr D$ and $F:\mathscr C\to \mathscr D$ be as in the above definition.
    \begin{enumerate}
        \item Assume that the functor $F$ preserves all sifted colimits. Then there exists a canonical equivalence of functors $\pi_0\circ \operatorname{Ani}(F)\simeq F\circ \pi_0\in \operatorname{Fun}(\operatorname{Ani}(\mathscr C),\mathscr D)$. (\cite[Subsection 5.1.4]{CS24}).
        \item If $F$ preserves sifted colimits (resp. small colimits), then the animation $\operatorname{Ani}(F)$ also preserves sifted colimits (resp. small colimits). (\cite[Corollary A.25]{Mao24}).
    \end{enumerate}
\end{remark}

\begin{lemma}\label{lem3.13}
    Let $\mathscr C$ (resp. $\mathscr D, \mathscr E$) be a category, $S$ (resp. $T$, $U$) be a set of compact projective generators of $\mathscr C$ (resp. $\mathscr D$, $\mathscr E$) and $\mathscr C_0$ (resp. $\mathscr D_0, \mathscr E_0$) be the full subcategory of $\mathscr C$ (resp. $\mathscr D, \mathscr E$) generated by $S$ (resp. $T$, $U$) under finite coproducts. Let $F:\mathscr C\to \mathscr D$ and $G:\mathscr D\to \mathscr E$ be functors satisfying either of the following conditions:\begin{enumerate}
        \item $F(\mathscr C_0)\subseteq \operatorname{Ind}(\mathscr D_0)$,
        \item $\operatorname{Ani}(G)(F(\mathscr C_0))\subseteq \mathscr E$.
    \end{enumerate}
    Then the natural transformation\[\operatorname{Ani}(G)\circ \operatorname{Ani}(F)\to \operatorname{Ani}(G\circ F)\] is an equivalence.
\end{lemma}
\begin{proof}
    This follows from Remark \ref{rem3.12}. For more detail, see \cite[Proposition 5.1.5]{CS24}.
    
\end{proof}
\begin{example}\label{exa3.14}
    \begin{enumerate}
        \item Let $[-]:\operatorname{Ani(Ring)}\to \mathcal S\simeq \operatorname{Ani(Set)}$ be the functor defined by \[R\mapsto \operatorname{Hom}_{\operatorname{Ani(Ring)}}(\mathbb Z[X],R).\] Following \cite[Appendix A]{BL22}, we call this $[R]$ the underlying space of $R$. Since this functor preserves all sifted colimits (by Remark \ref{rem3.4}), it coincides with the animation of the forgetful functor $\operatorname{Ring}\to \operatorname{Set}$. 
        \item Let $[-]:\operatorname{Ani(RingMod)}\to \mathcal S$ be the functor defined by \[(A,M)\mapsto \operatorname{Hom}_{\operatorname{Ani(RingMod)}}((\mathbb Z[X],\mathbb Z[X]),(A,M)).\] Since this functor preserves all sifted colimits, this coincides with the animation of the functor $\operatorname{RingMod}\to \operatorname{Set}:(A,M)\mapsto A\times M$. Furthermore, $[-]$ commutes with $\operatorname{RingMod}\to \operatorname{Set}:(A,M)\mapsto A\times M$.
        \item Let $[-]_1:\operatorname{Ani(RingMod)}\to \mathcal S$ be the functor defined by \[(A,M)\mapsto \operatorname{Hom}_{\operatorname{Ani(RingMod)}}((\mathbb Z[X],0),(A,M)).\] Since this functor preserves all sifted colimits, this coincides with the animation of the functor $\operatorname{RingMod}\to \operatorname{Set}:(A,M)\mapsto A$. Furthermore, $[-]_1$ commutes with $\operatorname{RingMod}\to \operatorname{Set}:(A,M)\mapsto A$.
        \item Let $[-]_2:\operatorname{Ani(RingMod)}\to \mathcal S$ be the functor defined by \[(A,M)\mapsto \operatorname{Hom}_{\operatorname{Ani(RingMod)}}((\mathbb Z,\mathbb Z),(A,M)).\] Since this functor preserves all sifted colimits, this coincides with the animation of the functor $\operatorname{RingMod}\to \operatorname{Set}:(A,M)\mapsto M$. Furthermore, $[-]_2$ commutes with $\operatorname{RingMod}\to \operatorname{Set}:(A,M)\mapsto M$.
        \item Let $\operatorname{Ani(RingMod)}\to \mathcal S\times \mathcal S$ be the functor induced by the functors $[-]_1$ and $[-]_2$ defined in (2) and (3). It follows from (2) and (3) that this functor preserves all sifted colimits. Thus, this functor coincides with the animation of the functor $\operatorname{RingMod}\to \operatorname{Set}\times\operatorname{Set}:(A,M)\mapsto (A,M)$. Since the set of objects $\{(\mathbb Z[X],0),(\mathbb Z,\mathbb Z)\}$ generates the $\infty$-category $\operatorname{Ani(RingMod)}$ under colimits, this functor is conservative.
    \end{enumerate}
\end{example}

The same is true for our category $\operatorname{RingMod^{p-tor}}$
\begin{example}\label{exa3.15}
    \begin{enumerate}
        \item Let $[-]:\operatorname{Ani(RingMod^{p-tor})}\to \mathcal S$ be the functor defined by \[(A,M)\mapsto \operatorname{Hom}_{\operatorname{Ani(RingMod^{p-tor})}}((\mathbb Z[X],\mathbb F_p[X]),(A,M)).\] Since this functor preserves all sifted colimits, this coincides with the animation of the functor $\operatorname{RingMod}\to \operatorname{Set}:(A,M)\mapsto A\times M$. Furthermore, $[-]$ commutes with $\operatorname{RingMod}\to \operatorname{Set}:(A,M)\mapsto A\times M$.
        \item Let $[-]_1:\operatorname{Ani(RingMod^{p-tor})}\to \mathcal S$ be the functor defined by \[(A,M)\mapsto \operatorname{Hom}_{\operatorname{Ani(RingMod^{p-tor})}}((\mathbb Z[X],0),(A,M)).\] Since this functor preserves all sifted colimits, this coincides with the animation of the functor \[\operatorname{RingMod^{p-tor}}\to \operatorname{Set}:(A,M)\mapsto A.\] Furthermore, $[-]_1$ commutes with $\operatorname{RingMod^{p-tor}}\to \operatorname{Set}:(A,M)\mapsto A$.
        \item Let $[-]_2:\operatorname{Ani(RingMod^{p-tor})}\to \mathcal S$ be the functor defined by \[(A,M)\mapsto \operatorname{Hom}_{\operatorname{Ani(RingMod^{p-tor})}}((\mathbb Z,\mathbb F_p),(A,M)).\] Since this functor preserves all sifted colimits, this coincides with the animation of the functor \[\operatorname{RingMod^{p-tor}}\to \operatorname{Set}:(A,M)\mapsto M.\] Furthermore, $[-]_2$ commutes with $\operatorname{RingMod^{p-tor}}\to \operatorname{Set}:(A,M)\mapsto M$.
        \item Let $\operatorname{Ani(RingMod^{p-tor})}\to \mathcal S\times \mathcal S$ be the functor induced by the functors $[-]_1$ and $[-]_2$ defined in (2) and (3). It follows from (2) and (3) that this functor preserves all sifted colimits. Thus, this functor coincides with the animation of the functor $\operatorname{RingMod^{p-tor}}\to \operatorname{Set}\times\operatorname{Set}:(A,M)\mapsto (A,M)$. Since the set of objects $\{(\mathbb Z[X],0),(\mathbb Z,\mathbb F_p)\}$ generates the $\infty$-category $\operatorname{Ani(RingMod^{p-tor})}$ under colimits, this functor is conservative.
    \end{enumerate}
\end{example}

Next, we see a relationship between the category $\operatorname{Ani(RingMod^{p-tor})}$ and derived $\infty$-categories. Before this, we define some notations.

\begin{definition}\label{def3.16}(\cite[Notation 25.2.1.1]{Lur18}).
    \begin{itemize}
         
        \item Let $\operatorname{CAlg(Sp)}$ be the $\infty$-category of $\mathbb E_{\infty}$-rings. Then there exists a canonical functor $\operatorname{Ani(Ring)}\to \operatorname{CAlg(Sp)}$ defined by animation. We call the image of an animated ring $A$ by this functor the underlying $\mathbb E_{\infty}$-ring of $A$ and denote it by $A^{\circ}$.
        \item Let $\operatorname{Mod(Sp)}$ be the $\infty$-category of pairs $(A,M)$, where $A$ is an $\mathbb E_{\infty}$-ring and $M$ is an $A$-module.
        \item Let $\operatorname{Mod(Sp)}_{\ge0}$ be the full subcategory of $\operatorname{Mod(Sp)}$ spanned by pairs $(A,M)$ such that $M$ is connective.
        \item Let $\operatorname{Mod(Sp)}_{\ge0}\to \operatorname{CAlg(Sp)}$ be the functor defined by $(A,M)\mapsto A$.
        \item Denote $\operatorname{SCRMod}_{\ge0}:=\operatorname{Ani(Ring)}\times_{\operatorname{CAlg(Sp)}}\operatorname{Mod(Sp)}_{\ge0}$.
    \end{itemize}
\end{definition}
We introduce similar notations:
\begin{definition}\label{def3.17}
    Let $\operatorname{Ani(Ring)}\to \operatorname{CAlg(Sp)}$ be the functor defined by the composition \[\operatorname{Ani(Ring)}\xrightarrow[]{A\mapsto A/^Lp}\operatorname{Ani(Ring)}\xrightarrow[]{A\mapsto A^{\circ}}\operatorname{CAlg(Sp)}.\]
    Let $\operatorname{Mod(Sp)}_{\ge0}\to \operatorname{CAlg(Sp)}$ be the functor defined by $(A,M)\mapsto A$. Using them, we define \[\operatorname{SCRMod^{p-tor}}_{\ge0}:=\operatorname{Ani(Ring)}\times_{\operatorname{CAlg(Sp)}}\operatorname{Mod(Sp)}_{\ge0}.\]
\end{definition}
\begin{prop}\label{pro3.18}
    There exists a canonical equivalence $\operatorname{SCRMod}_{\ge0}\simeq \operatorname{Ani(RingMod)}$ of $\infty$-categories compatible with the projections to the $\infty$-category $\operatorname{Ani(Ring)}$. In particular, the fiber of the projection $\operatorname{Ani(RingMod)}\to \operatorname{Ani(Ring)}$ on an animated ring $A\in \operatorname{Ani(Ring)}$ is equivalent to the connective part $\mathcal D_{\ge0}(A)$ of the derived $\infty$-category $\mathcal D(A)$.
\end{prop}
\begin{proof}
    This is \cite[Proposition 25.2.1.2]{Lur18}.
\end{proof}
The next proposition is a variant of Proposition \ref{pro3.18}.
\begin{prop}\label{pro3.19}
     There exists a canonical equivalence $\operatorname{SCRMod^{p-tor}}_{\ge0}\simeq \operatorname{Ani(RingMod^{p-tor})}$ of $\infty$-categories compatible with the projections to the $\infty$-category $\operatorname{Ani(Ring)}$. In particular, the fiber of the projection $\operatorname{Ani(RingMod)}\to \operatorname{Ani(Ring)}$ on an animated ring $A\in \operatorname{Ani(Ring)}$ is equivalent to the connective part $\mathcal D_{\ge0}(A/^Lp)$ of the derived $\infty$-category $\mathcal D(A/^Lp)$.
\end{prop}
\begin{proof}
    By \cite[Proposition. 5.5.3.12]{Lur09}, the $\infty$-category $\operatorname{SCRMod^{p-tor}}_{\ge0}$ is presentable, in particular admitting all sifted colimits. Thus, the functor \[f:\operatorname{RingMod^{p-tor}}\to \operatorname{SCRMod^{p-tor}}_{\ge0}:(A,M)\mapsto (A,M)\] induces \[F:\operatorname{Ani(RingMod^{p-tor})}\to \operatorname{SCRMod^{p-tor}}_{\ge0}\] by animation. 
    On the other hand, it follows from the description of colimits in the $\infty$-category defined by a fiber product of presentable $\infty$-categories, together with the fact that the forgetful functors $\operatorname{Ani(Ring)}\to \mathcal S$ and $\operatorname{Mod(Sp)}_{\ge0}\to \mathcal S$ commute with sifted colimits, that the functors \[\operatorname{SCRMod^{p-tor}}_{\ge0}\to \mathcal S:(A,M)\mapsto A\] and \[\operatorname{SCRMod^{p-tor}}_{\ge0}\to \mathcal S:(A,M)\mapsto M\] commute with all sifted colimits. This implies that objects $(\mathbb Z[X],0)$ and $(\mathbb Z,\mathbb F_p)$ are compact projective in $\operatorname{SCRMod^{p-tor}}_{\ge0}$. Therefore, the functor $f$ preserves compact projective objects. This observation and the fact that $f$ is fully faithful imply that $F$ is also fully faithful. Furthermore, it follows from  the description of colimits in the presentable $\infty$-category obtained by fiber products of presentable $\infty$-categories with presentable functors that $f$ preserves finite coproducts after restricting to the full subcategory spanned by compact projective objects. This implies that the functor $F$ preserves all small colimits. Therefore the adjoint functor theorem (\cite[Corollary 5.5.2.9]{Lur09}) implies that there exists a right adjoint $G$ of $F$. Since $F$ is fully faithful, it reduces to show that the functor $G$ is conservative. The composition \[\operatorname{SCRMod^{p-tor}}_{\ge0}\xrightarrow[]{G}\operatorname{Ani(RingMod^{p-tor})}\xrightarrow[]{([-]_1,[-]_2)}\mathcal S\times \mathcal S\] can be written as $(\operatorname{Hom}((\mathbb Z[X],0),G(-)),\operatorname{Hom}((\mathbb Z,\mathbb F_p),G(-)))$. By the right adjointness of $G$, this functor coincides with the functor $\operatorname{SCRMod^{p-tor}}_{\ge0}\to \mathcal S\times\mathcal S$ defined by $(A,M)\mapsto (A,M)$. Since this last functor is conservative, the functor $G$ is also conservative.  

    Next, we prove that the projections to the $\infty$-category $\operatorname{Ani(Ring)}$ is compatible under the identification $\operatorname{SCRMod^{p-tor}}_{\ge0}\simeq \operatorname{Ani(RingMod^{p-tor})}$. This follows from the fact that the composition \[\operatorname{RingMod^{p-tor,cp}}\xrightarrow[]{f}\operatorname{SCRMod^{p-tor}}_{\ge0}\xrightarrow[]{(A,M)\mapsto A}\operatorname{Ani(Ring)}\] coincides with the canonical functor $\operatorname{RingMod^{p-tor,cp}}\xrightarrow[]{(P,F)\mapsto P}\operatorname{Ani(Ring)}$, where $\operatorname{RingMod^{p-tor,cp}}$ is the full subcategory of $\operatorname{RingMod^{p-tor}}$ spanned by the pairs $(P,F)$ of a finitely generated polynomial ring $P$ and a finite free $P/p$-module $F$.
\end{proof}

\begin{remark}\label{rem3.20}
    By Proposition \ref{pro3.19} and Lemma \ref{lem3.8}, objects of the $\infty$-category $\operatorname{Ani(RingMod_*^{p-tor})}$ can be written as triples $(A,M,m)$, where $A$ is an animated $\mathbb Z_{(p)}$-algebra , $M$ is a connective $A/^Lp$-module (i.e. an object of $\mathcal D_{\ge 0}(A/^Lp)$) and $m$ is a map $\mathbb A/^Lp\to M$ of $\mathcal D_{\ge 0}(A/^Lp)$.
\end{remark}

\section{Definitions of Frobenius--Witt cotangent complexes}
Using the preceding preliminaries, we define the animation of Frobenius--Witt derivations.
\begin{definition}\label{def4.1}
    We denote by \[W^{\mathrm{an}}(-,-,-):\operatorname{Ani}(\operatorname{AlgMod^{p-tor}_*})\to \operatorname{Ani}(\operatorname{Ring})\] the animation of the functor
    \[W(-,-,-):\operatorname{AlgMod^{p-tor}_*}\to \operatorname{Ring}.\] 
\end{definition}
\begin{remark}\label{rem4.2}
    \begin{enumerate}
        \item We denote by $\mathrm{pr}_1:W^{\mathrm{an}}(A,M,m)\to R$ the natural transformation defined by the animation of the natural transformation $W(R,M,m)\to R:(a,x)\mapsto a$. Remark that this is a map of animated rings.
        \item We denote by $\mathrm{pr}_2:W^{\mathrm{an}}(A,M,m)\to M$ the natural transformation defined by the animation of the natural transformation $W(R,M,m)\to M:(a,x)\mapsto x$. Remark that this is only a map of spaces.
        \item Let $(\mathrm{pr}_1,\mathrm{pr}_2):W_2(R,M,m)\to R\times M$ be the map of spaces induced by $\mathrm{pr}_1$ and $\mathrm{pr}_2$. Since this functor preserves all sifted colimits by definition, it coincides with the animation of the functor \[\operatorname{AlgMod_*^{p-tor}}\to \operatorname{Set}:(A,M,m)\mapsto A\times M.\]  It follows from this observation and the canonical isomorphism $W_2(A,M,m)\xrightarrow[]{\cong}A\times M$ for each discrete object $(A,M,m)$ that there exists a natural equivalence \[W_2^{an}(A,M,m)\xrightarrow[]{\simeq } A\times M\] of spaces for all $(A,M,m)\in \operatorname{Ani(RingMod_*^{p-tor})}$.
    \end{enumerate}
\end{remark}
The following lemma is an analogue of \cite[Proposition A.3]{BL22}.
\begin{lemma}\label{lem4.3}
    Let $A$ be a $\mathbb Z_{(p)}$-algebra, $M$ be an $A/p$-module and $m$ be an element of $M$. Then there exists a canonical equivalence of animated rings\[W^{\mathrm{an}}_2(A,M,m)\xrightarrow[]{\simeq } W_2(A,M,m).\]
\end{lemma}
\begin{proof}
    By the universality of the left Kan extension, there exists a natural map $W_2^{an}(A,M,m)\to W_2(A,M,m)$ of animated rings. Then the composition \[W_2^{an}(A,M,m)\to W_2(A,M,m)\xrightarrow[]{\simeq}A\times M\] coincides with the canonical equivalence of spaces defined in Remark \ref{rem4.2}. Thus, we see that the map $W_2^{an}(A,M,m)\to W_2(A,M,m)$ is an equivalence of spaces. Since the forgetful functor $\operatorname{Ani(Ring)}\to \operatorname{Ani(Set)}\simeq \mathcal S$ is conservative, the map $W_2^{an}(A,M,m)\to W_2(A,M,m)$ is also an equivalence of animated rings.
\end{proof}
\begin{remark}\label{rem4.4}
    Lemma \ref{lem4.3} ensures that it is unambiguous to denote $W^{\mathrm{an}}(-,-,-)$ by $W_2(-,-,-)$.
\end{remark}

\begin{definition}\label{def4.5}
    Let $A$ be an animated $\mathbb Z_{(p)}$-algebra, $M$ be an connective $A/^Lp$-module (i.e. an object of the connective part of the derived $\infty$-category $\mathcal D(A/^Lp)$) and $m:A/^Lp\to M$ be a map of $A/^Lp$-modules. The space of Frobenius--Witt derivations $\operatorname{FWDer^{an}}(A,(M,m))\in \mathcal S$ from $A$ to $(M,m)$ is defined by \[\operatorname{Hom}_{\operatorname{Ani(Ring)_{/A}}}(A,W_2(A,M,m)),\] where the structure map $W_2(A,M,m)\to A$ is $\mathrm{pr}_1$. We call an element of the space $\operatorname{FWDer^{an}}(A,(M,m))$ an animated Frobenius--Witt derivation on $A$ valued in $(M,m)$.
\end{definition}
\begin{prop}\label{pro4.6}
    Let $A$ be a $\mathbb Z_{(p)}$-algebra, $M$ be a $A/p$-module and $m$ be an element of $M$. Then there exists a canonical equivalence of spaces\[\operatorname{FWDer}(A,(M,m))\to \operatorname{FWDer^{an}}(A,(M,m))\] In particular, the notion of animated Frobenius--Witt derivations coincides with the notion of Frobenius--Witt derivations defined in \cite{Sai22} in this case.
\end{prop}
\begin{proof}
    The former half follows from Lemma \ref{lem4.3} and the latter follows from the former and Proposition \ref{pro2.8}
\end{proof}
\begin{remark}
    By Proposition \ref{pro4.6}, it is unambiguous to denote $\operatorname{FWDer}^{\mathrm{an}}(A,(M,m))$ by $\operatorname{FWDer}(A,(M,m))$.
\end{remark}

\begin{prop}\label{pro4.7}
    Let $A$ be an animated $\mathbb Z_{(p)}$-algebra. The functor \[\mathcal D_{\ge0}(A/^Lp)_{(A/^Lp)/}\to \mathcal S:(M,m)\mapsto \operatorname{FWDer}(A,(M,m))\] is corepresented by an object $(F\mathbb L_A,\widetilde m)$ of $\mathcal D_{\ge0}(A/^Lp)_{(A/^Lp)/}$.
\end{prop}
\begin{proof}
    By \cite[Proposition 5.5.2.7]{Lur09} \cite[Tag 06L4]{Lur26} and \cite[Tag 06LA]{Lur26}, it suffices to show that the functor \[\mathcal D_{\ge0}(A/^Lp)_{(A/^Lp)/}\to \operatorname{Ani(Ring)}_{/A}:(M,m)\mapsto W_2(A,M,m)\] preserves all filtered colimits and all small limits. The former follows from the fact that the functor \[\operatorname{Ani(AlgMod_*^{p-tor})}\to \operatorname{Ani(Ring}):(A,M,m)\mapsto W_2(A,M,m)\] commutes with all filtered colimits. To finish the proof, it remains to prove the latter. Let $\{(M_i,m_i)\}_{i\in I}$ be an arbitrary small diagram in $\mathcal D_{\ge0}(A/^Lp)_{(A/^Lp)/}$. Then there exists a canonical map \[W_2(A,\operatorname{lim}_{i\in I}(M_i,m_i))\to \operatorname{lim}_{i\in I}W_2(A,M_i,m_i).\] it suffices to show that this map is equivalence as a map of spaces. This follows from the fact that this map can be identified with \[A\times \operatorname{lim}_{i\in I}M_i\to \operatorname{lim}_{i\in I}(A\times M_i)\] in $\mathcal S$ (see (3) of Remark \ref{rem4.2}) and the fact that all small limits commutes with products.
\end{proof}

\begin{definition}\label{def4.8}
    For an animated $\mathbb Z_{(p)}$-algebra $A$, we call the object $(F\mathbb L_A,\widetilde m)$ the Frobenius--Witt cotangent complex of $A$. We denote the universal Frobenius--Witt derivation by $w:A\to F\mathbb L_A$ and call it the Frobenius--Witt differential on $A$.
\end{definition}

\begin{prop}\label{pro4.8.1}
    The construction $A\mapsto (A,F\mathbb L_A,\widetilde m)$ defines a functor \[\operatorname{Ani(Ring)}_{\mathbb Z_{(p)}/}\to \operatorname{Ani(AlgMod_{*}^{p-tor})}:A\mapsto (A,F\mathbb L_A,\widetilde m)\] of $\infty$-categories.
\end{prop}
\begin{proof}
    Since the functor \[\operatorname{Ani(RingMod^{p-tor})_{(\mathbb Z_{(p)},0,0)/}}\to \operatorname{Ani(Ring)}_{\mathbb Z_{(p)}/}:(A,M)\mapsto A\] is a cartesian fibration, the functor \[p:\operatorname{Ani(AlgMod_*^{p-tor})}\to \operatorname{Ani(Ring)}_{\mathbb Z_{(p)}/}:(A,M,m)\mapsto A\] is also a cartesian cofibration by \cite[Tag, 01UR]{Lur26}. On the other hand, the functor \[q:\operatorname{Fun}(\Delta^1,\operatorname{Ani(Ring)}_{\mathbb Z_{(p)}/})\to \operatorname{Ani(Ring)}_{\mathbb Z_{(p)}/}:(A\to B)\mapsto B\] is also a cartesian fibration (whose cartesian edges are defined by fiber products). Consider the following commutative diagram of $\infty$-categories: \[ \begin{tikzpicture}[auto]
\node (a) at (0, 2) {$\operatorname{Ani(AlgMod_*^{p-tor})}$}; \node (x) at (6, 2) {$\operatorname{Fun}(\Delta^1,\operatorname{Ani(Ring)}_{\mathbb Z_{(p)}/})$};
\node (b) at (0, 0) {$\operatorname{Ani(Ring)}_{\mathbb Z_{(p)}/}$};   \node (y) at (6, 0) {$\operatorname{Ani(Ring)}_{\mathbb Z_{(p)}/}$};
\draw[->] (a) to node {$\scriptstyle G$} (x);
\draw[->] (x) to node {$\scriptstyle q$} (y);
\draw[->] (a) to node[swap] {$\scriptstyle p$} (b);
\draw[->] (b) to node[swap] {$\scriptstyle id $} (y);
\end{tikzpicture},\]
where we denote by $G$ the functor $(A,M,m)\mapsto W_2(A,M,m)$. Since we have $W_2(A,M,m)\simeq A\times M$ in $\mathcal S$, the functor $G$ sends $p$-cartesian edges to $q$-cartesian edges. By \cite[Proposition 7.3.2.6]{Lur17}, Proposition \ref{pro4.7} and the above observations, $G$ has a relative left adjoint $F$ in the sense of \cite[Definition 7.3.2.2]{Lur17}. Define a functor $F'$ by the following composition\[F':\operatorname{Ani(Ring)}_{\mathbb Z_{(p)}/}\xrightarrow[]{A\mapsto (A=A)} \operatorname{Fun}(\Delta^1,\operatorname{Ani(Ring)}_{\mathbb Z_{(p)}/})\xrightarrow[]{F}\operatorname{Ani(AlgMod_*^{p-tor})}.\] By considering the fiber over $A\in \operatorname{Ani(Ring)}_{\mathbb Z_{(p)}/}$, we see that the object $F'(A)$ corepresents the functor $\mathcal D(A/^Lp)_{(A/^Lp)/}\to \mathcal S:(M,m)\mapsto \operatorname{FWDer}(A,(M,m))$. Therefore, $F'$ is a functor of $\infty$-categories satisfying $F'(A)\simeq (A,F\mathbb L_A,\widetilde m)$.

\end{proof}

\begin{prop}\label{pro4.10.1}
    The functor \[\operatorname{Ani(Ring)}_{\mathbb Z_{(p)}/}\to \operatorname{Ani(AlgMod_{*}^{p-tor})}:A\mapsto (A,F\mathbb L_A,\widetilde m)\] commutes with all sifted colimits and pushouts.

    In particular, the functor 
    \[\operatorname{Ani(Ring)}_{\mathbb Z_{(p)}/}\to \mathcal D(\mathbb F_p)\] commutes with all sifted colimits and pushouts.
\end{prop}
\begin{proof}
Let $\{A_i\}_{i\in I}$ be an arbitrary sifted or pushout diagram in $\operatorname{Ani(Ring)}_{\mathbb Z_{(p)}/}$. Denote by $A$ its colimit. By functoriality (Proposition\ref{pro4.8.1}), there exists a canonical map $\operatorname{colim}_{i\in I}(F\mathbb L_{A_i}\otimes^L_{A_i}A)\to F\mathbb L_A$. For any $(M,m)\in \mathcal D(A/^Lp)_{(A/^Lp)/}$, we have the following equivalences in $\mathcal S$:\begin{align*}
    \operatorname{Hom}_{\mathcal D(A/^Lp)_{(A/^Lp)/}}(\operatorname{colim}_{i\in I}(F\mathbb L_{A_i}\otimes^L_{A_i}A),M)&\simeq \operatorname{lim}_{i\in I}\operatorname{Hom}_{\mathcal D(A/^Lp)_{(A/^Lp)/}}(F\mathbb L_{A_i}\otimes^L_{A_i}A,M)\\
    &\simeq \operatorname{lim}_{i\in I}\operatorname{FWDer}(A_i,M)\\
    &\simeq \operatorname{lim}_{i\in I}\operatorname{Hom}_{\operatorname{Ani(Ring)}_{\mathbb Z_{(p)}//A_i}}(A_i,W_2(A_i,M))\\
    &\simeq \operatorname{lim}_{i\in I}\operatorname{Hom}_{\operatorname{Ani(Ring)}_{\mathbb Z_{(p)}//A}}(A_i,W_2(A,M))\\
    &\simeq \operatorname{Hom}_{\operatorname{Ani(Ring)}_{\mathbb Z_{(p)}//A}}(\operatorname{colim}_{i\in I} A_i,W_2(A,M))\\
    &\simeq \operatorname{FWDer}(A,M)\\
    &\simeq \operatorname{Hom}_{\mathcal D(A/^Lp)_{(A/^Lp)/}}(F\mathbb L_A,M), 
\end{align*}
where, for simplicity of notation, we write $N$ for an object $(N,n)$ of $\mathcal D(A/^Lp)_{(A/^Lp)/}$. Remark that in the 4-th equivalence in the above computation, we used the fact that the following diagram of animated rings is a pullback square:
\[ \begin{tikzpicture}[auto]
\node (a) at (0, 2) {$W_2(A_i,M)$}; \node (x) at (4, 2) {$W_2(A,M)$};
\node (b) at (0, 0) {$A_i$};   \node (y) at (4, 0) {$A$};
\draw[->] (a) to node {$\scriptstyle $} (x);
\draw[->] (x) to node {$\scriptstyle $} (y);
\draw[->] (a) to node[swap] {$\scriptstyle $} (b);
\draw[->] (b) to node[swap] {$\scriptstyle $} (y);
\end{tikzpicture},\]
which follows from (3) of Remark \ref{rem4.2}. Thus we proved the proposition.
\end{proof}

\begin{prop}\label{pro5.2}
    Let $A:=\mathbb Z_{(p)}[X_1,\dots,X_r]$ be a polynomial $\mathbb Z_{(p)}$-algebra. Then there is a canonical equivalence $F\Omega_{A}\cong F\mathbb L_A$ in $\mathcal D(A/^Lp)$, which is a free $A/p$-module of rank $r+1$.
\end{prop}
\begin{proof}
    By \cite[Proposition 2.5.3]{Sai22}, there is a canonical isomorphism\[F\Omega_A\cong (A/p\cdot w(p))\oplus (\bigoplus_{i=1}^rA/p\cdot w(X_i)),\] where $w$ is the universal Frobenius--Witt derivation. For any object $(M,m)\in \mathcal D_{\ge 0}(A/^Lp)_{(A/^Lp)/}$, there exists a canonical map\[\operatorname{Hom}_{\mathcal D_{\ge 0}(A/^Lp)_{(A/^Lp)/}}((F\Omega_A,w(p)),(M,m))\xrightarrow[]{\alpha\mapsto \alpha\circ w}\operatorname{FWDer}(A,(M,m)).\]
    Let $F$ be a left adjoint functor of the forgetful functor $\operatorname{Ani(Ring)}_{\mathbb Z_{(p)}/}\to \operatorname{Ani(Set)}$. Then, $F(\{1,2,\dots,r\})\simeq \mathbb Z_{(p)}[X_1,\dots,X_r]$ holds. The space $\operatorname{FWDer}(A,(M,m))$ can be canonically identified with the space of the following commutative diagrams in $\operatorname{Ani(Ring)}_{\mathbb Z_{(p)}/}$ :
    \[\begin{tikzpicture}[auto]
\node (a) at (0, 2) {$\mathbb Z_{(p)}[X_1,\dots,X_r]$}; \node (x) at (6, 2) {$W_2(A,M,m)$};
\node (b) at (0, 0) {$\mathbb Z_{(p)}[X_1,\dots,X_r]$};   \node (y) at (6, 0) {$\mathbb Z_{(p)}[X_1,\dots,X_r]$};
\draw[->] (a) to node {$\scriptstyle $} (x);
\draw[->] (x) to node {$\scriptstyle $} (y);
\draw[->] (a) to node[swap] {$\scriptstyle \mathrm{id}$} (b);
\draw[->] (b) to node[swap] {$\scriptstyle \mathrm{id}$} (y);
\end{tikzpicture}.\]
By adjunction, this space is equivalent to the space of the following commutative diagrams in $\mathcal S$

\[\begin{tikzpicture}[auto]
\node (a) at (0, 2) {$\{1,2,\dots,r\}$}; \node (x) at (6, 2) {$[W_2(A,M,m)]$};
\node (b) at (0, 0) {$\{1,2,\dots,r\}$};   \node (y) at (6, 0) {$[A]$};
\draw[->] (a) to node {$\scriptstyle $} (x);
\draw[->] (x) to node {$\scriptstyle $} (y);
\draw[->] (a) to node[swap] {$\scriptstyle \mathrm{id}$} (b);
\draw[->] (b) to node[swap] {$\scriptstyle i\mapsto X_i$} (y);
\end{tikzpicture}.\]
This space is equivalent to $M^{\times r}$ by Remark \ref{rem4.2}. Combining these observations, we see that the map
\[\operatorname{Hom}_{\mathcal D_{\ge 0}(A/^Lp)_{(A/^Lp)/}}((F\Omega_A,w(p)),(M,m))\xrightarrow[]{\alpha\mapsto \alpha\circ w}\operatorname{FWDer}(A,(M,m))\]
can be identified with $M^{\times r}\xrightarrow[]{id}M^{\times r}$. This implies that the object $(F\Omega_A,w(p))$ satisfies the universal property of $(F\mathbb L_A,\widetilde m)$. This implies that there exists a canonical equivalence $F\Omega_A\simeq F\mathbb L_A$ in $\mathcal D_{\ge0}(A/p)$.
\end{proof}

\begin{coro}\label{cor4.11.1}
    The functor \[\operatorname{Ani(Ring)}_{\mathbb Z_{(p)}/}\to \mathcal D_{\ge0}(\mathbb F_p):A\mapsto F\mathbb L_A\] coincides with the animation of the functor \[\operatorname{Ring}_{\mathbb Z_{(p)}/}\to \operatorname{Mod}(\mathbb F_p):A\mapsto F\Omega_A.\]
\end{coro}
\begin{proof}
    This follows from Proposition \ref{pro4.10.1} and Proposition \ref{pro5.2}.
\end{proof}

Next, we define the Frobenius--Witt cotangent complex in another way, and prove that these definitions are equivalent.

\begin{lemma}\label{lem4.9}
   The functor \[\operatorname{AlgMod_*^{p-tor}}\to \operatorname{Ring}_{\mathbb Z_{(p)}}:(A,M,m)\mapsto W_2(A,M,m)\] has a left adjoint given by \[\operatorname{Ring}_{\mathbb Z_{(p)}}\to \operatorname{AlgMod_*^{p-tor}}:A\mapsto (A,F\Omega_A,\widetilde m),\] where $m$ is the image of $p$ by the universal Frobenius--Witt derivation.
\end{lemma}
\begin{proof}
    First, remark that the canonical surjection \[W_2(A,M,m)\xrightarrow[]{\mathrm{pr}_1}A\] has a square-zero kernel, so the ring $W_2(A,M,m)$ is a $\mathbb Z_{(p)}$-algebra for all $(A,M,m)\in \operatorname{AlgMod_*^{p-tor}}$.

    Next, we prove that there exists a natural bijection \[\operatorname{Hom}_{\operatorname{Ring}_{\mathbb Z_{(p)}}}(A,W_2(B,N,n))\cong \operatorname{Hom}_{\operatorname{AlgMod_*^{p-tor}}}((A,F\Omega_A,\widetilde m),(B,N,n))\] for any $A\in \operatorname{Ring}_{\mathbb Z_{(p)}}$ and $(B,N,n)\in \operatorname{AlgMod_*^{p-tor}}$. For any ring homomorphism $\phi:A\to W_2(B,N,n)$, let $\varphi:A\to B$ be the composition $A\to W_2(B,N,n)\xrightarrow[]{\mathrm{pr}_1}B$, let $\alpha:F\Omega_A\to N$ be the $\varphi$-semilinear map induced from the composition $A\to W_2(B,N,n)\xrightarrow[]{\mathrm{pr}_2}N$ by the universal property of $F\Omega_A$. Then we get a natural map\[\operatorname{Hom}_{\operatorname{Ring}_{\mathbb Z_{(p)}}}(A,W_2(B,N,n))\to \operatorname{Hom}_{\operatorname{AlgMod_*^{p-tor}}}((A,F\Omega_A,\widetilde m),(B,N,n)):\phi\mapsto (\varphi,\alpha).\]

    On the other hand, Let $(\varphi,\alpha)$ be an element of $\operatorname{Hom}_{\operatorname{AlgMod_*^{p-tor}}}((A,F\Omega_A,\widetilde m),(B,N,n))$. define \[\phi:A\to W_2(B,N,n):a\mapsto (\varphi(a),\alpha (w(a))),\] where $w:A\to F\Omega_A$ is the universal Frobenius--Witt derivation. It is straightforward to see that the map $\phi$ is a ring homomorphism, so we get a natural map \[\operatorname{Hom}_{\operatorname{AlgMod_*^{p-tor}}}((A,F\Omega_A,\widetilde m),(B,N,n))\to \operatorname{Hom}_{\operatorname{Ring}_{\mathbb Z_{(p)}}}(A,W_2(B,N,n)):(\varphi,\alpha)\mapsto \phi.\]

    It is straightforward to check that these maps between mapping sets are inverse to each other.
\end{proof}
\begin{prop}\label{pro4.10}
    The functor \[W_2(-,-,-):\operatorname{Ani}(\operatorname{AlgMod^{p-tor}_*})\to \operatorname{Ani}(\operatorname{Ring}_{\mathbb Z_{(p)}}):(A,M,m)\mapsto W_2(A,M,m)\] has a left adjoint $A\mapsto (A,F\mathbb L_A,\widetilde m)$.
\end{prop}
\begin{proof}
    This follows from \cite[Corollary 2.3]{Mao24}, Lemma \ref{lem4.9} and Corollary \ref{cor4.11.1}.
\end{proof}

\begin{remark}
    By Corollary \ref{cor4.11.1} and Proposition \ref{pro4.10}, we see that the Frobenius--Witt cotangent complex $F\mathbb L_A$ for $A\in \operatorname{Ani(Ring)}_{\mathbb Z_{(p)}/}$ can be constructed by the following equivalent ways:\begin{itemize}
        \item Define $F\mathbb L_A$ as an object corepresenting the functor \[(M,m)\mapsto \operatorname{FWDer}(A,(M,m)).\]
        \item Define $F\mathbb L_A$ by the animation of the functor \[\operatorname{Ring}_{\mathbb Z_{(p)}}\to \operatorname{AlgMod_*^{p-tor}}:A\mapsto (A,F\Omega_A,w(p)).\]
        \item Define $F\mathbb L_A$ by the animation of the functor \[\operatorname{Ring}_{\mathbb Z_{(p)}}\to \operatorname{Mod}(\mathbb F_p):A\mapsto F\Omega_A.\]
        \item Define $F\mathbb L_A$ as an left adjoint of the functor \[W_2(-,-,-):\operatorname{Ani}(\operatorname{AlgMod^{p-tor}_*})\to \operatorname{Ani}(\operatorname{Ring}_{\mathbb Z_{(p)}}):(A,M,m)\mapsto W_2(A,M,m).\]
    \end{itemize}
\end{remark}

\section{Properties of Frobenius--Witt cotangent complexes}
%In this section, we show that properties analogous to $F\Omega_A$, proved in \cite{Sai22}, also hold for $F\mathbb L_A$.
In this section, we prove the basic properties of Frobenius--Witt cotangent complexes. The most important among them is the existence of fundamental fiber sequence (Theorem \ref{thm5.5}), which prolongs the right exact sequence \[F\Omega_A\otimes_AB\to F\Omega_B\to F^*(\Omega_{B/A}\otimes_AA/p)\to 0\] proved in \cite[Lemma 2.9]{Sai22}. Many other properties, as well as the computations in the next section, follow from it. We also prove that the functor $A\mapsto F\mathbb L_A$ is an fpqc sheaf (Theorem \ref{thm5.11}), which may be regarded as an analogue of the fact that the cotangent complex forms an fpqc sheaf.

\begin{prop}\label{pro5.1}
    Let $A$ be an animated $\mathbb Z_{(p)}$-algebra. Then there exists a canonical isomorphism $\pi_0(F\mathbb L_A)\xrightarrow[]{\simeq} F\Omega_{\pi_0(A)}$ of $\pi_0(A)/p$-modules.
\end{prop}
\begin{proof}
    It suffices to show that the ring $\pi_0(F\mathbb L_A)$ also satisfies the universal property of $F\Omega_{\pi_0(A)}$. This follows from the following canonical equivalences of spaces:
    \begin{align*}
    \operatorname{Hom}_{\operatorname{Ring_{/A}}}(A,W_2(A,M,m))&\simeq \operatorname{Hom}_{\operatorname{Ani(Ring)_{/A}}}(A,W_2(A,M,m))\\&\simeq \operatorname{Hom}_{\mathcal D(A/^Lp)_{(A/^Lp)/}}((F\mathbb L_A,\widetilde m),(M,m))\\&\simeq \operatorname{Hom}_{\operatorname{Mod}(\pi_0(A)/p)}((\pi_0(F\mathbb L_A),\widetilde m),(M,m)), 
    \end{align*}
   for any $\pi_0(A/^Lp)=\pi_0(A)/p$-module $M$ and any element $m\in M$.
\end{proof}

%\begin{prop}\label{pro5.3}
   % The functor $F\mathbb L_{(-)}:\operatorname{Ani(Ring)}_{\mathbb Z_{(p)}}\to \mathcal D_{\ge0}(\mathbb F_p)$ preserves all sifted colimits.
%\end{prop}
%\begin{proof}
 %   This follows from Corollary \ref{cor4.11}.
%\end{proof}

\begin{lemma}\label{lem5.4}
    The category $\operatorname{Fun}(\Delta^1,\operatorname{Ani(Ring)})_{/(\mathbb Z_{(p)}=\mathbb Z_{(p)})}$ of maps of animated $\mathbb Z_{(p)}$-algebras has a set of compact projective generators $\{\mathbb Z_{(p)}[X]\to \mathbb Z_{(p)}[X,Y]|\,X,Y\in \operatorname{FinSet}\}$.
\end{lemma}
\begin{proof}
    This follows from the argument preceding \cite[Definition 2.33]{Mao24}.
\end{proof}
The next theorem is very useful to endow fundamental properties of Frobenius--Witt cotangent complexes.
\begin{theorem}\label{thm5.5}
    Let $A\to B$ be a morphism of animated $\mathbb Z_{(p)}$-algebras. Then there is a canonical equivalence\[\operatorname{Cofib}(F\mathbb L_A\otimes_A^LB\to F\mathbb L_B)\xrightarrow[]{\simeq }F^*(\mathbb L_{(B/^Lp)/(A/^Lp)})\in \mathcal D_{\ge0}(\mathbb F_p),\] where $F^*$ is the derived pull-back along the absolute Frobenius map $F:B/^Lp\to B/^Lp$.
\end{theorem}
\begin{proof}
    By Proposition \ref{pro4.10.1}, the both sides commutes with all sifted colimits. By this observation and Lemma \ref{lem5.4}, it suffices to show the statement in the case that $A\to B$ is the canonical inclusion $\mathbb Z_{(p)}[X]\to \mathbb Z_{(p)}[X,Y]$, where $X$ and $Y$ are finite sets. In this case, the equivalence follows from \cite[Proposition 2.10.1]{Sai22}.
\end{proof}

\begin{coro}\label{cor5.6}
    Let $A\to B$ be a morphism of animated $\mathbb Z_{(p)}$-algebras such that $\mathbb L_{B/A}\otimes_A^LA/^Lp$ vanishes. Then there exists a canonical equivalence \[F\mathbb L_A\otimes^L_AB\xrightarrow[]{\simeq }F\mathbb L_B.\]
\end{coro}
\begin{proof}
    This follows from Theorem \ref{thm5.5} and the fact that $\mathbb L_{(B/^Lp)/(A/^Lp)}\simeq \mathbb L_{(B/A)}\otimes_{\mathbb Z_{(p)}}^L\mathbb F_p$ holds.
\end{proof}
\begin{remark}\label{rem5.7}
    The condition $\mathbb L_{B/A}\otimes_A^LA/^Lp\simeq 0$ holds in either of the following cases.\begin{itemize}
        \item $A\to B$ is an etale map of $\mathbb Z_{(p)}$-algebras.
        \item $A\to B$ is a derived base change of a map between perfectoid rings. (To see this, we may assume that $A$ and $B$ are perfectoid. Since $B\simeq A_{\mathrm{inf}}(B)\otimes^L_{A_{\mathrm{inf}}(A)}A$ holds, where $A_{\mathrm{inf}}(A)\to A$ is the Fontaine's $\theta$-map (for definition, see \cite[Section 3.1]{BMS18}), the conclusion follows from $\mathbb L_{A_{\mathrm{inf}}(B)/A_{\mathrm{inf}}(A)}\otimes_{A_{\mathrm{inf}}(A)}A_{\mathrm{inf}}(A)/p\simeq \mathbb L_{B^{\flat}/A^{\flat}}\simeq 0$, where $A^{\flat}$ is the tilt of $A$).
    \end{itemize}
\end{remark}
\begin{coro}\label{cor5.8}
    Let $A$ be an animated $\mathbb Z_{(p)}$-algebra and $\widehat A$ be its derived $p$-adic completion. Then the canonical map $F\mathbb L_A\to F\mathbb L_{\widehat A}$ is an equivalence.
\end{coro}
\begin{proof}
    By the above theorem and the equivalence $A/^Lp\xrightarrow[]{\simeq }\widehat A/^Lp$, it suffices to show the equivalence of the canonical map $F\mathbb L_A\to F\mathbb L_A\otimes_A^L\widehat A$. This follows from the following calculation:\begin{align*}
        F\mathbb L_A\otimes_A^L\widehat A&\simeq F\mathbb L_A\otimes_{A/^Lp}A/^Lp\otimes_A^L\widehat A\\&\simeq F\mathbb L_A\otimes_{A/^Lp}A/^Lp\\&\simeq F\mathbb L_A.
    \end{align*}
\end{proof}

\begin{coro}\label{cor5.9}
    Let $A$ be a smooth $\mathbb Z_{(p)}$-algebra. Then there is a canonical equivalence \[F\mathbb L_A\xrightarrow[]{\simeq}F\Omega_A.\]
\end{coro}
\begin{proof}
    This follows from Corollary \ref{cor5.6} and Proposition \ref{pro5.2}.
\end{proof}
\begin{coro}\label{cor5.10}
    Let $A$ be an animated $\mathbb Z_{(p)}$-algebra. Then there exists a canonical equivalence\[\operatorname{Cofib}(A/^Lp\xrightarrow[]{w(p)}F\mathbb L_A)\xrightarrow[]{\simeq }F^*(\mathbb L_{(A/^Lp)/\mathbb F_p})\]
\end{coro}
\begin{proof}
    This follows from the isomorphism $F\Omega_{\mathbb Z_{(p)}}\cong \mathbb F_p\cdot w(p)$ and Theorem \ref{thm5.5} by taking $A=\mathbb Z_{(p)}$ and $B=A$.
\end{proof}
The following result is one of the main theorems of this paper. The proof is similar to \cite[Theorem 3.1]{BMS19}.
\begin{theorem}\label{thm5.11}
    The functor $F\mathbb L_{(-)}:\operatorname{Ani(Ring)}_{\mathbb Z_{(p)}}\to \mathcal D_{\ge0}(\mathbb F_p)$ is an fpqc sheaf. (Being an fpqc sheaf means that, for any faithfully flat map $A\to B$ of animated $\mathbb Z_{(p)}$-algebras, the canonical map $F\mathbb L_{A}\to \operatorname{lim}_{\Delta}F\mathbb L_{(B^{\bullet})}$ is an equivalence, where $B^{\bullet}$ is the Čech conerve of $A\to B$).
\end{theorem}
\begin{proof}
    By Theorem \ref{thm5.5}, there exists a natural cofiber sequence of cosimplicial animated modules:\[F\mathbb L_A\otimes_A^LB^{\bullet}\to F\mathbb L_{B^{\bullet}}\to F^*(\mathbb L_{B^{\bullet}/A}\otimes_{\mathbb Z_{(p)}}^L\mathbb F_p)\to \] in $\mathcal D(\mathbb F_p)$, where $F$ is the absolute Frobenius map on $B^{\bullet}/^Lp$. By the fpqc descent for modules, there exists a natural equivalence $F\mathbb L_A\simeq \operatorname{lim}_{\Delta}(B^{\bullet}\otimes^L_{A}F\mathbb L_{A})$. Thus, it suffices to show the vanishing of $\operatorname{lim}_{\Delta}F^*(\mathbb L_{B^{\bullet}/A}\otimes_{\mathbb Z_{(p)}}^L\mathbb F_p)$. By the base change property of cotangent complexes and the fact that the $B^{\bullet}/^Lp$ coincides with the Čech conerve of the faithfully flat map $A/^Lp\to B/^Lp$, it suffices to show the vanishing of $\operatorname{lim}_{\Delta}F_{B^\bullet}^*\mathbb L_{B^{\bullet}/A}$ for any faithfully flat map $A\to B$ of animated $\mathbb F_p$-algebras. By the convergence of the Postnikov filtration and the commutativity of limits, the following equivalence holds:
    \begin{align*}
        \operatorname{lim}_{\Delta}F_{B^\bullet}^*\mathbb L_{B^{\bullet}/A}&\simeq \operatorname{lim}_{\Delta}(\operatorname{lim}_m(\tau _{\le m}F_{B^\bullet}^*\mathbb L_{B^{\bullet}/A}))\\
        &\simeq \operatorname{lim}_{m}(\operatorname{lim}_{\Delta}(\tau_{\le m}F^*_{B^\bullet}\mathbb L_{B^\bullet/A})).
    \end{align*}
    On the other hand, we have the following fiber sequence:
    \[\operatorname{lim}_{\Delta}(\tau_{\le (m+1)}F^*_{B^\bullet}\mathbb L_{B^\bullet/A})\to \operatorname{lim}_{\Delta}(\tau_{\le m}F^*_{B^\bullet}\mathbb L_{B^\bullet/A})\to \operatorname{lim}_{\Delta}((\pi_mF^*_{B^\bullet}\mathbb L_{B^\bullet/A})[m])\to .\] Combining them together with the following obvious vanishing:
    \[\tau_{\le -1}F^*_{B^\bullet}\mathbb L_{B^\bullet/A}\simeq 0,\]
    the assertion reduces to showing the following vanishing:
    \[\operatorname{lim}_{\Delta}(\pi_mF^*_{B^\bullet}\mathbb L_{B^\bullet/A})\simeq 0\in \mathcal D(\mathbb Z)\] for any $m\ge 0$. In other words, it suffices to show that the cosimplicial abelian group
    \[\pi_m(F^*_{B^\bullet}\mathbb L_{B^\bullet/A})\] is acyclic for all $m\ge 0$. By the faithful flatness of $A\to B$, it suffices to show the acyclicity of the following cosimplicial module:
    \[\pi_m(F^*_{B^\bullet}\mathbb L_{B^\bullet/A})\otimes_{\pi_0(A)}\pi_0(B)\simeq \pi_m(F^*_{B^\bullet}\mathbb L_{B^\bullet/A}\otimes_A^LB).\]
    Write $C^{\bullet}:=B^{\bullet}\otimes_AB$, which coincides with the \v{C}ech conerve of the map $B\to B\otimes_A^LB:=C$. Then there are following equivalences:\begin{align*}
        (F^*_{B^{\bullet}}\mathbb L_{B^{\bullet}/A})\otimes^L_AB&\simeq \mathbb L_{B^{\bullet}/A}\otimes^L_{B^{\bullet}}C^{\bullet}\otimes^L_{C^{\bullet},F_{C^{\bullet}}}C^{\bullet}\\&\simeq \mathbb L_{C^{\bullet}/B}\otimes^L_{C^{\bullet},F_{C^{\bullet}}}C^{\bullet}.
    \end{align*}
    Let $F:\operatorname{Ani(Ring)}_{B/}\to \operatorname{Ab}$ be the functor defined to be $F(B'):=\pi_m(\mathbb L_{B'/B}\otimes^L_{B',F_{B'}}B')$, then there exists a canonical equivalence \[F(B)\xrightarrow[]{\simeq }\operatorname{lim}_{\Delta}F(C^{\bullet})\in \mathcal D(\mathbb Z)\] because $B\to C$ has a section. By this observation and the fact that $F(B)=\mathbb L_{B/B}\otimes^L_{B,F_{B}}B\simeq 0$, we get the conclusion.
    
\end{proof}

Next, we see a relationship between Frobenius--Witt cotangent complexes and derived completions. The following proposition is inspired by a talk of R. Takeuchi.

\begin{prop}\label{pro5.11.1}
    Let $A$ be an animated $\mathbb Z_{(p)}$-algebra, $I$ be a finitely generated ideal of $\pi_0(A)$. Then there is a canonical equivalence \[(F\mathbb L_A)_I^{\wedge}\simeq (F\mathbb L_{A_I^{\wedge}})_I^{\wedge},\] where $(-)_I^{\wedge}$ is the derived $I$-completion functor.
\end{prop}
\begin{proof}
    By induction on the number of generators of $I$, we may assume that $I$ is generated by an element $f\in \pi_0(A)$. Consider the fundamental fiber sequence\[F\mathbb L_{A}\otimes^L_{A}A_f^{\wedge}\to  F\mathbb L_{A_{f}^{\wedge}}\to F^*(\mathbb L_{A_f^{\wedge}/A}\otimes_A^LA/^Lp)\to .\] Taking derived $f$-completion of this fiber sequence, we have the following fiber sequence:
    \[(F\mathbb L_{A}\otimes^L_{A}A_f^{\wedge})_f^{\wedge}\to  (F\mathbb L_{A_{f}^{\wedge}})_f^{\wedge}\to (F^*(\mathbb L_{A_f^{\wedge}/A}\otimes_A^LA/^Lp))_f^{\wedge}\to .\]
    Since $A_f^{\wedge}/^Lf\simeq A/^Lf$ holds, derived Nakayama's lemma implies that the canonical map $(F\mathbb L_A)_f^{\wedge}\to (F\mathbb L_A\otimes_AA_f^{\wedge})_f^{\wedge}$ is an equivalence. Therefore, it suffices to show $(F^*(\mathbb L_{A_f^{\wedge}/A}\otimes_A^LA/^Lp))_f^{\wedge}\simeq 0$. For simplicity of notations, we write $B:=A/^Lp$. Then its derived quotient by $f$ can be computed as follows.
    \begin{align*}
        \mathbb (L_{B_f^{\wedge}/B}\otimes^L_{B_f^{\wedge},F}B_f^{\wedge})\otimes^L_{B_f^{\wedge}}B/^Lf&\simeq \mathbb L_{B_f^{\wedge}/B}\otimes^L_{B_f^{\wedge},F}B/^Lf\\
        &\simeq \mathbb L_{B_f^{\wedge}/B}\otimes^L_{B_f^{\wedge}}B/^Lf\otimes^L_{B/^Lf,F}B/^Lf\\
        &\simeq \mathbb L_{B_f^{\wedge}/B}\otimes^L_{B_f^{\wedge}}(B_f^{\wedge}\otimes_B^LB/^Lf)\otimes^L_{B/^Lf,F}B/^Lf\\
        &\simeq (\mathbb L_{B_f^{\wedge}/B}\otimes_B^LB/^Lf)\otimes^L_{B/^Lf,F}B/^Lf\\
        &\simeq \mathbb L_{(B/^Lf)/(B/^Lf)}\otimes^L_{B/^Lf,F}B/^Lf\\
        &\simeq 0.
    \end{align*}
    Thus the conclusion follows from derived Nakayama's lemma.
\end{proof}
\begin{prop}\label{pro5.11.2}
    Let $C$ be an animated $\mathbb Z_{(p)}$-algebra. Then there is a canonical equivalence   \[F\mathbb L_{C[x_1,\dots,x_r]}\simeq (F\mathbb L_C\otimes_C^LC[x_1,\dots,x_r])\oplus \bigoplus_{i=1}^{r}(C[x_1,\dots,x_r]/p\cdot w(x_i)).\]
\end{prop}
\begin{proof}
    This follows by applying Proposition \ref{pro4.10.1} to the following pushout square of animated rings:
    \[ \begin{tikzpicture}[auto]
\node (a) at (0, 2) {$\mathbb Z_{(p)}$}; \node (x) at (4, 2) {$C$};
\node (b) at (0, 0) {$\mathbb Z_{(p)}[x_1,\dots,x_r]$};   \node (y) at (4, 0) {$C[x_1,\dots,x_r]$};
\draw[->] (a) to node {$\scriptstyle $} (x);
\draw[->] (x) to node {$\scriptstyle $} (y);
\draw[->] (a) to node[swap] {$\scriptstyle $} (b);
\draw[->] (b) to node[swap] {$\scriptstyle $} (y);
\end{tikzpicture}.\]
\end{proof}

%\begin{prop}\label{pro5.11.2}
 %   Let $R$ be a $p$-torsionfree regular noetherian ring. Then there is a canonical equivalence $F\mathbb L_R\simeq F\Omega_R$. 
%\end{prop}
%\begin{proof}
    %Consider the fundamental fiber sequence (Theorem \ref{thm5.5}):
    %\[F\mathbb L_{\mathbb Z_{(p)}}\otimes_{\mathbb Z_{(p)}}^LR\to F\mathbb L_{R}\to F^*(\mathbb L_{(R/p)/\mathbb F_p})\to .\]
    %associated to $\mathbb Z_{(p)}\to R$.
%\end{proof}

\section{Vanishing for perfectoid rings and an application to the regularity criterion}
In this section, we present several explicit computations of Frobenius--Witt cotangent complexes. In particular, we give two proofs of the vanishing of them for perfectoid rings. The first proof proceeds by decomposing a perfectoid ring as a fiber product of a $p$-torsion-free perfectoid ring and a perfect ring, and then reducing to these two cases. The second proof gives a direct argument using the fact that the functor of Frobenius--Witt cotangent complexes preserves pushouts. Finally, we study a relationship between regularity and the Frobenius--Witt cotangent complex based on the results established so far.

\begin{prop}\label{pro6.1}
    Let $A$ be a perfect $\delta$-ring. Then there is a canonical equivalence $F\mathbb L_A\simeq A/p\cdot w(p)$.
\end{prop}
\begin{proof}
    By \cite[Lemma 2.28]{BS22}, $A$ is $p$-torsionfree, so $\mathbb L_{A/\mathbb Z_{(p)}}\otimes_{\mathbb Z_{(p)}}^L\mathbb F_p\simeq \mathbb L_{(A/p)/\mathbb F_p}\simeq 0$, where the last equivalence follows from the fact that $A/p$ is a perfect $\mathbb F_p$-algebra. By this observation and Theorem \ref{thm5.5}, the conclusion follows from the following calculation: \begin{align*}
        F\mathbb L_{A}&\simeq F\mathbb L_{\mathbb Z_{(p)}}\otimes^L_{\mathbb Z_{(p)}}A\\&\simeq (\mathbb F_p\cdot w(p))\otimes^L_{\mathbb Z_{(p)}}A\\&\simeq A/p\cdot w(p).
    \end{align*}
\end{proof}

\begin{coro}\label{cor6.2}
    Let $R$ be a $p$-torsionfree perfectoid ring. Then $F\mathbb L_R\simeq 0$ holds.
\end{coro}
\begin{proof}
    Let $\theta:A_{\mathrm{inf}}(R)\to R$ be the Fontaine's $\theta$-map defined in \cite[Section 3.1]{BMS18} and $I$ be the kernel of $\theta$. Applying the theorem to this map, we obtain the following fiber sequence\[F\mathbb L_{A_{\mathrm{inf}}(R)}\otimes_{A_{\mathrm{inf}}(R)}^LR\to F\mathbb L_R\to F^*(\mathbb L_{R/A_{\mathrm{inf}}(R)}\otimes_{R}^LR/p)\to. \] By the above proposition and the equivalence $\mathbb L_{R/A_{\mathrm{inf}}(R)}\simeq I/I^2[1]$, where $I/I^2$ is a free $R$-module of rank one, we have equivalences $H_i(F\mathbb L_R)\simeq 0$ for all $i\ge2$ and the following exact sequence\[0\to H_1(F\mathbb L_R)\to F^*(I/I^2\otimes_RR/p)\xrightarrow[]{w'} R/p\cdot w(p)\to F\Omega_R\to 0.\] On the other hand, we have the following exact sequence by \cite[Proposition 2.3.2]{Sai22}\[F^*(I/I^2\otimes_RR/p)\xrightarrow[]{w} R/p\cdot w(p)\to F\Omega_R\to 0,\]where $w$ is the map induced from the universal Frobenius--Witt derivation $A_{\mathrm{inf}}(R)\to F\Omega_{A_{\mathrm{inf}}}:a\mapsto \delta(a)w(p)$ (this follows from \cite[Lemma 1.3.2]{Sai22}). Here, we denote by $\delta$ the canonical $\delta$-structure on the ring of Witt vectors $A_{\mathrm{inf}}(R)=W(R^{\flat})$. Let $\xi$ be a generator of the ideal $I$. Then $\delta(\xi)$ is a unit in $A_{\mathrm{inf}}(R)$, so the map $w$ is surjective. This implies $F\Omega_R\cong 0$. This implies that the map $w'$ in the longer exact sequence is a surjection between free $R/p$-modules of rank one. Thus, $w'$ is an isomorphism, which implies $H_1(F\mathbb L_R)\cong 0$.
    
\end{proof}

Next, we compute the Frobenius--Witt cotangent complexes for $\mathbb F_p$-algebras. The key computation is the case of $\mathbb F_p$ (Lemma\ref{lem6.3}).

\begin{lemma}\label{lem6.3}
    $F\mathbb L_{\mathbb F_p}\simeq 0$ holds.
\end{lemma}
\begin{proof}
    We see the ring $\mathbb Z_{(p)}$ as an algebra over the ring of polynomials $\mathbb Z_{(p)}[t]$ by $\mathbb Z_{(p)}[t]\to \mathbb Z_{(p)}:t\mapsto 0$. Let $P_{\bullet}\to \mathbb Z_{(p)}$ be the bar resolution of $\mathbb Z_{(p)}$ as a $\mathbb Z_{(p)}[t]$-algebra as in \cite[Construction 4.13]{Iye07}. By construction, each term $P_n$ of $P_{\bullet}$ is a polynomial ring of the form $P_n=\mathbb Z_{(p)}[t,t_1,\dots,t_n]$. Thus, the base change $Q_{\bullet}:=P_{\bullet}\otimes_{\mathbb Z_{(p)}[t]}\mathbb Z_{(p)}[t]/(t-p)$ is a simplicial resolution of $\mathbb F_p\simeq \mathbb Z_{(p)}\otimes^L_{\mathbb Z_{(p)}[t]}\mathbb Z_{(p)}[t]/(t-p)$ as a $\mathbb Z_{(p)}$-algebra, whose $n$-th term is of the form $Q_n=\mathbb Z_{(p)}[t_1,\dots,t_n]$. By \cite[Construction 4.13]{Iye07}, the face maps in the lower degrees can be written as follows:\begin{itemize}
        \item $d^2_0:\mathbb Z_{(p)}[t_1,t_2]\to \mathbb Z_{(p)}[t_1]:t_1\mapsto p, t_2\mapsto t_1$,
        \item $d^2_1:\mathbb Z_{(p)}[t_1,t_2]\to \mathbb Z_{(p)}[t_1]:t_1\mapsto t_1, t_2\mapsto t_1$,
        \item $d^2_2:\mathbb Z_{(p)}[t_1,t_2]\to \mathbb Z_{(p)}[t_1]:t_1\mapsto t_1, t_2\mapsto 0$,
        \item $d^1_0:\mathbb Z_{(p)}[t_1]\to \mathbb Z_{(p)}:t_1\mapsto p$,
        \item $d^1_1:\mathbb Z_{(p)}[t_1]\to \mathbb Z_{(p)}:t_1\mapsto 0$.
    \end{itemize}
    Thus, the lower terms of the complex associated to the simplicial abelian group $F\Omega_{Q_{\bullet}}\otimes_{Q_{\bullet}}\mathbb Z_{(p)}$ are as follows:

    \begin{align*}    
    \to \mathbb F_p\cdot w(p)\oplus \mathbb F_p\cdot w(t_1)\oplus \mathbb F_p\cdot w(t_2)\xrightarrow[]{\tiny{\begin{cases} w(p)\mapsto w(p)\\w(t_1)\mapsto w(p)\\w(t_2)\mapsto 0\end{cases}}}\mathbb F_p\cdot w(p)\oplus \mathbb F_p\cdot w(t_1)\xrightarrow[]{\tiny{\begin{cases}w(p)\mapsto 0\\w(t_1)\mapsto w(p)\end{cases}}}\mathbb F_p\cdot w(p).
    \end{align*}
    Thus, we have \[\pi_1(F\mathbb L_{\mathbb F_p})\cong H_1(F\Omega_{Q_{\bullet}}\otimes_{Q_{\bullet}}\mathbb Z_{(p)})\cong 0.\]

    On the other hand, we have $\pi_0(F\mathbb L_{\mathbb F_p})\cong F\Omega_{\mathbb F_p}\cong F^*\Omega_{\mathbb F_p/\mathbb Z_{(p)}}\cong 0$, here we used Proposition \ref{pro5.1} in the first isomorphism and \cite[Corollary 2.4.2]{Sai22} in the second isomorphism.

    Next, we prove $\pi_i(F\mathbb L_{\mathbb F_p})\cong 0$ for any $i\ge0$. We have proved this for $i=0$ and $i=1$ in the previous paragraphs. To see the vanishing for the higher degree, we use the fundamental fiber sequence (Theorem \ref{thm5.5}) associated to $\mathbb Z_{(p)}\to \mathbb F_p$:\[F\mathbb L_{\mathbb Z_{(p)}}\otimes^L_{\mathbb Z_{(p)}}\mathbb F_p\to F\mathbb L_{\mathbb F_p}\to F^*(\mathbb L_{\mathbb F_p/\mathbb Z_{(p)}}\otimes_{\mathbb Z_{(p)}}^L\mathbb F_p)\to .\]
    Since $F\mathbb L_{\mathbb Z_{(p)}}\simeq \mathbb F_p\cdot w(p)$, the first term $F\mathbb L_{\mathbb Z_{(p)}}\otimes^L_{\mathbb Z_{(p)}}\mathbb F_p$ is concentrated in homological degree $[0,1]$ and whose $0$-th homotopy group and $1$-st homotopy group are $\mathbb F_p$.  Since $p$ is a non-zerodivisor in $\mathbb Z_{(p)}$, the last term $F^*(\mathbb L_{\mathbb F_p/\mathbb Z_{(p)}}\otimes_{\mathbb Z_{(p)}}^L\mathbb F_p)$ is concentrated in the homological degree $[1,2]$ and whose $1$-st homotopy group and $2$-nd homotopy group are $\mathbb F_p$. Combining them, the associated long exact sequence is as follows:

\[\begin{tikzcd}[column sep=small,row sep=small]
0 \arrow[r] &
\pi_3(F\mathbb L_{\mathbb F_p}) \arrow[r] &
0 \arrow[dll,] \\
0 \arrow[r] &
\pi_2(F\mathbb L_{\mathbb F_p}) \arrow[r] &
\mathbb F_p \arrow[dll,] \\
\mathbb F_p \arrow[r] &
0 \arrow[r] &
\mathbb F_p \arrow[dll,] \\
\mathbb F_p \arrow[r] &
0 \arrow[r] &
0.
\end{tikzcd}
\]

This implies $\pi_i(F\mathbb L_{\mathbb F_p})\cong 0$ for all $i\ge2$. Thus, we have $F\mathbb L_{\mathbb F_p}\simeq 0$.

\end{proof}

By the above computation, we see that Frobenius--Witt cotangent complexes for rings of characteristic $p$ can be written using the Frobenius maps and cotangent complexes.
\begin{coro}\label{cor6.4}
    Let $R$ be an animated $\mathbb F_p$-algebra. Then we have a canonical equivalence \[F\mathbb L_{R}\simeq F^*(\mathbb L_{R/\mathbb F_p}\otimes^L(\mathbb F_p/^Lp)).\]
\end{coro}
\begin{proof}
    This follows from applying Lemma \ref{lem6.3} to the fundamental fiber sequence (Theorem \ref{thm5.5}) associated to the structure map $\mathbb F_p\to R$. 
\end{proof}
\begin{coro}\label{cor6.5}
    Let $R$ be a perfect $\mathbb F_p$-algebra. Then $F\mathbb L_{R}\simeq 0$ holds.
\end{coro}
\begin{proof}
    By Corollary \ref{cor6.4}, we have $F\mathbb L_R\simeq F^*(\mathbb L_{R/\mathbb F_p}\otimes^L\mathbb F_p/^Lp)$. Since $\mathbb L_{R/\mathbb F_p}\simeq 0$ holds for any perfect $\mathbb F_p$-algebra $R$, we get the conclusion.
\end{proof}
Using Corollary \ref{cor6.5}, we can give a more conceptual proof to a special case of \cite[Corollary 4.12]{Sai22}.

\begin{coro}\label{cor6.5.1}
    Let $R$ be a local ring with perfect residue field $k$ of characteristic $p$. Then there exists a canonical isomorphism\[F\Omega_R\otimes_Rk\cong \pi_1(\mathbb L_{k/R}).\]
\end{coro}
\begin{proof}
    By Theorem \ref{thm5.5}, there is a canonical fiber sequence:\[F\mathbb L_R\otimes_R^Lk\to F\mathbb L_k\to F^*(\mathbb L_{k/R}\otimes_R^LR/^Lp)\to.\]
    Since the middle term vanishes by Corollary \ref{cor6.5}, we have \[F\mathbb L_R\otimes_R^Lk\simeq F^*(\mathbb L_{k/R}\otimes_k^Lk/^Lp)[-1].\] Taking the first homotopy groups on the both sides, we get \[F\Omega_R\otimes_Rk\cong F^*(\pi_1(\mathbb L_{k/R}))\cong \pi_1(\mathbb L_{k/R}),\] here we used the perfectness of $k$ in the second isomorphism.
\end{proof}

Next, we prove the vanishing of the Frobenius--Witt cotangent complex for general perfectoid rings using Corollary \ref{cor6.2} and Corollary \ref{cor6.5}. To do this, we first prove the following lemma. The author is grateful to R. Ishizuka for explaining the proof of this lemma.
\begin{lemma}\label{lem6.6}
    Let $R$ be a perfectoid ring. Let $\overline R$ be the quotient of $R$ by $p$-torsion elements and $(R/p)_{\mathrm{red}}$ be the reduction of $R/p$. Then $\overline{R}$ is a $p$-torsionfree perfectoid ring and $(R/p)_{\mathrm{red}}$ is a perfect $\mathbb F_p$-algebra. Furthermore, the functor \[\mathcal{D}(R)\to \mathcal D(\overline{R}\times (R/p)_{\mathrm{red}}):M\mapsto M\otimes^L_R(\overline{R}\times (R/p)_{\mathrm{red}})\]
   is conservative.
\end{lemma}
\begin{proof}
    The former assertion is a well known fact and one can find this in \cite[\S 2.1.3]{CS24}. We prove the latter assertion. By \cite[\S 2.1.3]{CS24}, the canonical map $R\to \overline{R}\times_{(\overline{R}/p)_{\mathrm{red}}} (R/p)_{\mathrm{red}}$ is isomorphic. This implies that the following sequence is exact.\[0\to R\to (\overline{R}\times (R/p)_{\mathrm{red}})\to (\overline{R}/p)_{\mathrm{red}}\to 0.\] Let $M$ be an arbitrary object of $\mathcal D(R)$ such that $M\otimes_R^L(\overline{R}\times (R/p)_{\mathrm{red}})\simeq 0$ holds. On the other hand, the above exact sequence implies the following fiber sequence:\[M\to M\otimes^L_R(\overline{R}\times (R/p)_{\mathrm{red}})\to M\otimes^L_R(\overline{R}/p)_{\mathrm{red}}\to .\] By our assumption on $M$, the middle term and the right term are equivalent to $0$. Thus, we have $M\simeq 0$.
\end{proof}
\begin{theorem}\label{thm6.7}
    Let $R$ be a perfectoid ring. Then $F\mathbb L_R\simeq 0$ holds.
\end{theorem}
\begin{proof}
    By Lemma \ref{lem6.6}, it suffices to prove $F\mathbb L_R\otimes^L_R\overline{R}\simeq 0$ and $F\mathbb L_R\otimes^L_R(R/p)_{\mathrm{red}}\simeq 0$. 
    
    We first prove the former assertion. Since $\overline{R}$ is a $p$-torsionfree perfectoid ring (Lemma \ref{lem6.6}), we have $F\mathbb L_{\overline R}\simeq 0$ by Corollary \ref{cor6.2}. By Corollary \ref{cor5.6} and Remark \ref{rem5.7}, we have $F\mathbb L_{R}\otimes_R^L\overline{R}\simeq F\mathbb L_{\overline R}$. Combining them, we get $F\mathbb L_R\otimes_R^L\overline{R}\simeq 0$. 

    For the latter assertion, the same argument reduces to showing $F\mathbb L_{(R/p)_{\mathrm{red}}}\simeq 0$. This was proved in Corollary \ref{cor6.5}.
\end{proof}

\begin{proof}[Another proof of Theorem \ref{thm6.7}]
    Let $\theta:A_{\mathrm{inf}}(R)\to R$ be the Fontaine's $\theta$-map. Since a generator $\xi\in A_{\mathrm{inf}}(R)$ of $\operatorname{ker}\theta$ is a non-zerodivisor, there is the following pushout square of animated $\mathbb Z_{(p)}$-algebras.
    \[ \begin{tikzpicture}[auto]
\node (a) at (0, 2) {$\mathbb Z_{(p)}[x]$}; \node (x) at (4, 2) {$A_{\mathrm{inf}}(R)$};
\node (b) at (0, 0) {$\mathbb Z_{(p)}$};   \node (y) at (4, 0) {$R$};
\draw[->] (a) to node {$\scriptstyle x\mapsto \xi$} (x);
\draw[->] (x) to node {$\scriptstyle $} (y);
\draw[->] (a) to node[swap] {$\scriptstyle x\mapsto 0$} (b);
\draw[->] (b) to node[swap] {$\scriptstyle $} (y);
\end{tikzpicture}.\]
Applying Proposition \ref{pro4.10.1} to this pushout square and using Proposition \ref{pro5.2} and Proposition \ref{pro6.1}, we have the following pushout square in $\mathcal D(R/^Lp)$:
\[\begin{tikzpicture}[auto]
\node (a) at (0, 2) {$((R/^Lp)\cdot w(p))\oplus ((R/^Lp)\cdot w(x))$}; \node (x) at (8, 2) {$(R/^Lp)\cdot w(p)$};
\node (b) at (0, 0) {$(R/^Lp)\cdot w(p)$};   \node (y) at (8, 0) {$F\mathbb L_R$};
\draw[->] (a) to node {$\scriptstyle \tiny{\begin{cases}
    w(p)\mapsto w(p)\\
    w(x)\mapsto w(\xi)=\delta(\xi)w(p)
\end{cases}}$} (x);
\draw[->] (x) to node {$\scriptstyle $} (y);
\draw[->] (a) to node[swap] {$\scriptstyle \tiny{\begin{cases}w(p)\mapsto w(p)\\w(x)\mapsto 0 \end{cases}}$} (b);
\draw[->] (b) to node[swap] {$\scriptstyle $} (y);
\end{tikzpicture},\] where in the upper horizontal map, we used the equality $w(\xi)=\delta(\xi)w(p)$, which follows from \cite[Lemma 1.3]{Sai22}. On the other hand, we have the following pushout square in $\mathcal D(R/^Lp)$: 
\[\begin{tikzpicture}[auto]
\node (a) at (0, 2) {$(R/^Lp)\cdot w(x)$}; \node (x) at (8, 2) {$((R/^Lp)\cdot w(p))\oplus ((R/^Lp)\cdot w(x))$};
\node (b) at (0, 0) {$0$};   \node (y) at (8, 0) {$(R/^Lp)\cdot w(p)$};
\draw[->] (a) to node {$\scriptstyle w(x)\mapsto w(x)$} (x);
\draw[->] (x) to node {$\scriptstyle $} (y);
\draw[->] (a) to node[swap] {$\scriptstyle x\mapsto 0$} (b);
\draw[->] (b) to node[swap] {$\scriptstyle $} (y);
\end{tikzpicture}.\]

Combining these pushout diagrams, we get another pushout diagram in $\mathcal D(R/^Lp)$:
\[\begin{tikzpicture}[auto]
\node (a) at (0, 2) {$(R/^Lp)\cdot w(x)$}; \node (x) at (6, 2) {$(R/^Lp)\cdot w(p)$};
\node (b) at (0, 0) {$0$};   \node (y) at (6, 0) {$F\mathbb L_R$};
\draw[->] (a) to node {$\scriptstyle w(x)\mapsto \delta(\xi)w(p)$} (x);
\draw[->] (x) to node {$\scriptstyle $} (y);
\draw[->] (a) to node[swap] {$\scriptstyle $} (b);
\draw[->] (b) to node[swap] {$\scriptstyle $} (y);
\end{tikzpicture}.\]

Since the upper horizontal map is induced by \[A_{\mathrm{inf}}(R)/p\xrightarrow[]{\times \delta(\xi)}A_{\mathrm{inf}}(R)/p\] and the element $\delta(\xi)\in A_{\mathrm{inf}}(R)$ is invertible, the upper horizontal map is an equivalence. Therefore, we get the conclusion.
 \end{proof}

 In the next two corollaries, we compute the Frobenius--Witt cotangent complexes of certain quotients of perfectoid rings. These computations will be applied to the regularity criterion for mixed characteristic complete noetherian $p$-torsionfree local rings in Corollary \ref{cor6.12}.
 
\begin{coro}\label{cor6.8}
    Let $R$ be a perfectoid ring. Let $f_1,\dots,f_r$ be a Koszul regular sequence in $R$. Let $S:=R/(f_1,\dots,f_r)$. Then there exists a canonical equivalence \[F\mathbb L_S\simeq (\bigoplus_{i=1}^r(S/^Lp)\cdot w(x_i))[1].\]
\end{coro}
\begin{proof}
    Since the sequence $f_1,\dots,f_r$ is Koszul regular, the following diagram is a pushout diagram of animated $\mathbb Z_{(p)}$-algebras:
    \[\begin{tikzpicture}[auto]
\node (a) at (0, 2) {$\mathbb Z_{(p)}[x_1,\dots,x_r]$}; \node (x) at (4, 2) {$\mathbb Z_{(p)}$};
\node (b) at (0, 0) {$R$};   \node (y) at (4, 0) {$S$};
\draw[->] (a) to node {$\scriptstyle x_i\mapsto 0$} (x);
\draw[->] (x) to node {$\scriptstyle $} (y);
\draw[->] (a) to node[swap] {$\scriptstyle x_i\mapsto f_i$} (b);
\draw[->] (b) to node[swap] {$\scriptstyle $} (y);
\end{tikzpicture}.\]

By Proposition \ref{pro4.10.1}, Proposition \ref{pro5.2} and Theorem \ref{thm6.7}, the above pushout diagram induces the following pushout diagram in $\mathcal D(S/^Lp)$:
 \[\begin{tikzpicture}[auto]
\node (a) at (0, 2) {$(S/^Lp)\cdot w(p)\oplus \bigoplus_{i=1}^r(S/^Lp)\cdot w(x_i)$}; \node (x) at (6, 2) {$(S/^Lp)\cdot w(p)$};
\node (b) at (0, 0) {$0$};   \node (y) at (6, 0) {$F\mathbb L_S$};
\draw[->] (a) to node {$\scriptstyle \mathrm{pr}$} (x);
\draw[->] (x) to node {$\scriptstyle $} (y);
\draw[->] (a) to node[swap] {$\scriptstyle $} (b);
\draw[->] (b) to node[swap] {$\scriptstyle $} (y);
\end{tikzpicture}.\]
This implies $F\mathbb L_S\simeq (\bigoplus_{i=1}^r(S/^Lp)\cdot w(x_i))[1]$.
\end{proof}

%The author wonders more generally whether $F\mathbb L_S$, regarded as an object of $\mathcal{D}(S/^Lp)$, has flat amplitude concentrated in homological degree 1 for any quasiregular semiperfectoid ring $S$.

\begin{coro}\label{cor6.9}
    Let $S$ be a $p$-torsionfree quasiregular semiperfectoid ring. Then the Frobenius--Witt cotangent complex $F\mathbb L_S\in \mathcal D(S/p)$ has  Tor-amplitude in homological degree $1$.
\end{coro}
For the definitions of the notions of quasiregular semiperfectoid, see \cite[Definition 4.20]{BMS19}.
\begin{proof}
    Since $S$ is quasiregular semiperfectoid, there is a surjective ring map $R\twoheadrightarrow S$ from a perfectoid ring $R$. Consider the fundamental fiber sequence (Theorem \ref{thm5.5}) associated to this map:\[F\mathbb L_R\otimes_R^LS\to F\mathbb L_S\to F^*(\mathbb L_{S/R}\otimes_R^LR/^Lp)\to .\] The right term of this fiber sequence can be computed as follows:\begin{align*}
        F^*(\mathbb L_{S/R}\otimes_R^LR/^Lp)&\simeq F^*(\mathbb L_{S/R}\otimes_R^L(R\otimes^L_{\mathbb Z_{(p)}}\mathbb F_p))\\
        &\simeq F^*(\mathbb L_{S/R}\otimes^L_{\mathbb Z_{(p)}}\mathbb F_p)\\
        &\simeq F^*(\mathbb L_{S/R}\otimes_S^L(S\otimes^L_{\mathbb Z_{(p)}}\mathbb F_p))\\
        &\simeq F^*(\mathbb L_{S/R}\otimes_S^LS/p),
    \end{align*} where we used the $p$-torsionfreeness of $S$ in the last equivalence. By \cite[Lemma 4.25]{BMS19}, this object has Tor amplitude in homological degree $1$ as an object of $\mathcal D(S/p)$. On the other hand, the left term of the above fiber sequence is zero by Theorem \ref{thm6.7}. Combining them, we get the conclusion.
\end{proof}

Next, we see a relationship between the regularity of mixed-characteristic complete noetherian local rings and their Frobenius--Witt cotangent complexes. The assumption of completeness will be removed later (Theorem \ref{thm6.16}).

\begin{theorem}\label{thm6.10}
    Let $R$ be a $p$-torsionfree complete noetherian regular local ring whose residue field is characteristic $p$. Then the Frobenius--Witt cotangent complex $F\mathbb L_R$ of $R$ is concentrated in degree zero.
\end{theorem}
\begin{proof}
    There exists a quasi-syntomic cover $R\to S$ with $S$ perfectoid by Lemma \ref{lem6.10.1} below. Let $S':=S\widehat \otimes_RS$ be the $p$-completed tensor product. By \cite[Lemma 4.16]{BMS19}, $S'$ is a $p$-torsionfree quasiregular semiperfectoid ring. Applying Proposition \ref{pro4.10.1} to the following pushout diagram of animated $\mathbb Z_{(p)}$-algebras
\[\begin{tikzpicture}[auto]
\node (a) at (0, 2) {$R$}; \node (x) at (2, 2) {$S$};
\node (b) at (0, 0) {$S$};   \node (y) at (2, 0) {$S'$};
\draw[->] (a) to node {$\scriptstyle $} (x);
\draw[->] (x) to node {$\scriptstyle $} (y);
\draw[->] (a) to node[swap] {$\scriptstyle $} (b);
\draw[->] (b) to node[swap] {$\scriptstyle $} (y);
\end{tikzpicture}, \]
we get the following pushout diagram in $\mathcal D(\mathbb F_p)$:
\[\begin{tikzpicture}[auto]
\node (a) at (0, 2) {$F\mathbb L_R\otimes_{R/p}^LS'/p$}; \node (x) at (3, 2) {$0$};
\node (b) at (0, 0) {$0$};   \node (y) at (3, 0) {$F\mathbb L_{S'}$};
\draw[->] (a) to node {$\scriptstyle $} (x);
\draw[->] (x) to node {$\scriptstyle $} (y);
\draw[->] (a) to node[swap] {$\scriptstyle $} (b);
\draw[->] (b) to node[swap] {$\scriptstyle $} (y);
\end{tikzpicture},\]
where we used Theorem \ref{thm6.7}. Thus, we have \[F\mathbb L_{R}\otimes_{R/p}^LS'/p\simeq F\mathbb L_{S'}[-1].\] Since $S'$ is a $p$-torsionfree quasiregular semiperfectoid ring, Corollary \ref{cor6.9} implies that the object $F\mathbb L_{R}\otimes_{R/p}^LS'/p$ is concentrated in homological degree $0$. Combining this and the faithful flatness of $R/p\to S'/p$, we get the conclusion.    
\end{proof}
The next lemma was used above. This lemma was originally formulated under the assumption that the residue field is perfect, in which case it follows immediately from \cite[Example 3.8]{BIM19}. The author is grateful to R. Ishizuka for explaining how to remove this assumption. 
\begin{lemma}\label{lem6.10.1}
    Let $R$ be a complete noetherian regular local ring of residue characteristic $p$. Then there is a quasi-syntomic map $R\to S$ with $S$ perfectoid.
\end{lemma}
\begin{proof}
    By Cohen's structure theorem, there is an isomorphism\[R\cong C\llbracket x_1,\dots,x_r\rrbracket/(p-f),\] where $C$ is the Cohen ring of the residue field $k$ of $R$ and $f$ is an element of the ideal $(x_1,\dots,x_r)\subseteq C\llbracket x_1,\dots,x_r\rrbracket$. Then $C\llbracket x_1,\dots,x_r\rrbracket$ has the $\delta$-structure extended from a Frobenius lift on $C$ defined by $\delta(x_i)=0$. By \cite[Example 5.2]{IN26}, the pair $(A,I):=(C\llbracket x_1,\dots,x_r\rrbracket,(p-f))$ is a prism. Let $(A_{\mathrm{perf}},IA_{\mathrm{perf}})$ be the perfection of $(A,I)$ as a prism. Then the quotient $S:=A_{\mathrm {perf}}/IA_{\mathrm{perf}}$ is a perfectoid ring. To prove that the map $R\to S$ defined by the base change of $A\to A_{\mathrm{perf}}$ is quasisyntomic, it suffices to show that the map $A\to A_{\mathrm{perf}}$ is quasisyntomic. Since the class of quasisyntomic $A$-algebras are closed under derived $p$-completed filtered colimits, it suffices to show that the Frobenius lift $F$ on $A$ associated to the $\delta$-structure on $A$ is quasi-syntomic.  Denote by $F_C$ the Frobenius lift on $C$. Then we can write $F$ by the composition:\[A=C\llbracket x_1,\dots,x_r\rrbracket\xrightarrow[]{F_C\widehat \otimes\mathrm{id_{\mathbb Z\llbracket x_1,\dots,x_r\rrbracket}}}C\llbracket x_1,\dots,x_r\rrbracket\xrightarrow[]{x_i\mapsto x_i^p}C\llbracket x_1,\dots,x_r\rrbracket\]
    Therefore it suffices to show that the map $F_C$ and $\mathbb Z_p\llbracket x_1,\dots,x_r\rrbracket\xrightarrow[]{x_i\mapsto x_i^p}\mathbb Z_p\llbracket x_1,\dots,x_r\rrbracket$ are quasi-syntomic. First, we treat with the first map. By definition of Cohen rings, it suffices to show that the cotangent complex $\mathbb L_{\varphi_*k/k}$ has Tor amplitude in homological degree $[0,1]$, where $\varphi:k\to k$ is the Frobenius map and $\varphi_*k$ is the $k$-algebra $k$ whose structure map is $\varphi$. Consider the following transitive fiber sequence associated to $\mathbb F_p\to k\to \varphi_*k$:
    \[\mathbb L_{k/\mathbb F_p}\otimes_k^L\varphi_*k\to \mathbb L_{\varphi_*k/\mathbb F_p}\to \mathbb L_{\varphi_*k/k}\to .\] Since $\mathbb F_p$ is perfect and $k$ is regular, $\mathbb L_{k/\mathbb F_p}$ is a flat $k$-module. Since the Frobenius map kills differential forms, the first map $\mathbb L_{k/\mathbb F_p}\otimes_k^L\varphi_*k\to \mathbb L_{\varphi_*k/\mathbb F_p}$ in the above fiber sequence is a zero map. Combining them, we see that the complex $\mathbb L_{\varphi_*k/k}$ has Tor amplitude in homological degree $[0,1]$. Finally, we treat with the second map $\mathbb Z_p\llbracket x_1,\dots,x_r\rrbracket\xrightarrow[]{x_i\mapsto x_i^p}\mathbb Z_p\llbracket x_1,\dots,x_r\rrbracket$. To prove that this map is quasi-syntomic, it remains to prove that the complex $\mathbb L_{\varphi_*\mathbb F_p\llbracket x_1,\dots,x_r\rrbracket/\mathbb F_p\llbracket x_1,\dots,x_r\rrbracket}$ has Tor amplitude in homological degree $[0,1]$. The proof of this assertion is similar to the one for $\mathbb L_{\varphi_*k/k}$. Consider the transitive fiber sequence associated to $\mathbb F_p\to \mathbb F_p\llbracket x_1,\dots,x_r\rrbracket\to \varphi_*\mathbb F_p\llbracket x_1,\dots,x_r\rrbracket$:
    \[\mathbb L_{\mathbb F_p\llbracket x_1,\dots,x_r\rrbracket/\mathbb F_p}\otimes_{\mathbb F_p\llbracket x_1,\dots,x_r\rrbracket}^L\varphi_*{\mathbb F_p\llbracket x_1,\dots,x_r\rrbracket}\to \mathbb L_{\varphi_*\mathbb F_p\llbracket x_1,\dots,x_r\rrbracket/\mathbb F_p}\to \mathbb L_{\varphi_*\mathbb F_p\llbracket x_1,\dots,x_r\rrbracket/\mathbb F_p\llbracket x_1,\dots,x_r\rrbracket}\to\]
Since $\mathbb F_p$ is a perfect field and $\mathbb F_p\llbracket x_1,\dots,x_r\rrbracket$ is a regular ring, the complex $\mathbb L_{\mathbb F_p\llbracket x_1,\dots,x_r\rrbracket/\mathbb F_p}$ is a flat $\mathbb F_p\llbracket x_1,\dots,x_r\rrbracket$-module. From this and the fact that the first map in the above fiber sequence is zero, we get the conclusion.
\end{proof}

\begin{remark}\label{rem6.11}\begin{enumerate}
    %\item  The author wonders whether the above theorem holds for any noetherian regular local ring of mixed characteristic. 
    \item For a complete regular noetherian local ring $R$ of characteristic $p$, the conclusion of Theorem \ref{thm6.10} fails. Indeed, let $R:=\mathbb F_p\llbracket x\rrbracket$ be the ring of formal power series with coefficients in $\mathbb F_p$. Then Corollary \ref{cor6.4} implies \[F\mathbb L_R\simeq F^*(\mathbb L_{R/{\mathbb F_p}}\otimes^L_{\mathbb F_p}\mathbb F_p/^Lp).\] By Popescu's theorem (\cite[Tag 07GC]{Sta26}), the cotangent complex $\mathbb L_{R/\mathbb F_p}$ is flat over $R$. Thus, the right hand side is not concentrated in degree zero.
    \item Let $R$ be a quasisyntomic ring. Then the Frobenius--Witt cotangent complex $F\mathbb L_R$ has Tor amplitude in homological degree $[0,1]$. This follows from the fundamental fiber sequence (Theorem \ref{thm5.5}) associated to $\mathbb Z_p\to R$.
\end{enumerate}
\end{remark}

The following corollary is a derived variant of Saito's regularity criterion (\cite[Theorem 3.4]{Sai22}). The assumption of completeness will be removed later (Corollary \ref{cor6.17}).
\begin{coro}\label{cor6.12}
Let $R$ be a $p$-torsionfree complete noetherian local ring of residue characteristic $p$. Assume that the ring $R/p$ is $F$-finite. Then the following conditions are equivalent:\begin{enumerate}
        \item $R$ is regular,
        \item $F\mathbb L_R$ is a free $R/p$-module of rank $\operatorname{dim}R+\operatorname{log}_p[k:k^p]$,
    \end{enumerate}
    where $k$ is the residue field of $R$.
\end{coro}
\begin{proof}
    %Though this follows from Theorem \ref{thm6.10} and \cite[Theorem 4.5]{Tak26}, which is a generalization of \cite[Theorem 3.4]{Sai22}, we give another proof of the implication $(1)\Rightarrow (2)$ in terms of Frobenius--Witt cotangent complexes.
    $(1)\Rightarrow (2)$. By Lemma \ref{lem6.10.1} and its proof, there is a prism $(A,I):=(C\llbracket x_1,\dots,x_r\rrbracket,(p-f))$ such that $R\cong A/I$. Let $(A_{\mathrm{perf}},IA_{\mathrm{perf}})$ be the perfection of $(A,I)$ as a prism. Write $S:=A_{\mathrm{perf}}/IA_{\mathrm{perf}}$, which is perfectoid. Then we have the following pushout square of animated $\mathbb Z_{(p)}$-algebras:
    \[\begin{tikzpicture}[auto]
\node (a) at (0, 2) {$A$}; \node (x) at (2, 2) {$R$};
\node (b) at (0, 0) {$A_{\mathrm{perf}}$};   \node (y) at (2, 0) {$S$};
\draw[->] (a) to node {$\scriptstyle $} (x);
\draw[->] (x) to node {$\scriptstyle $} (y);
\draw[->] (a) to node[swap] {$\scriptstyle $} (b);
\draw[->] (b) to node[swap] {$\scriptstyle $} (y);
\end{tikzpicture}. \]
Applying Proposition \ref{pro4.10.1} to this, we get the following pushout square:
\[\begin{tikzpicture}[auto]
\node (a) at (0, 2) {$F\mathbb L_A\otimes_A^LS$}; \node (x) at (4, 2) {$F\mathbb L_R\otimes_R^LS$};
\node (b) at (0, 0) {$F\mathbb L_{A_{\mathrm{perf}}}\otimes^L_{A_{\mathrm{perf}}}S$};   \node (y) at (4, 0) {$F\mathbb L_S$};
\draw[->] (a) to node {$\scriptstyle $} (x);
\draw[->] (x) to node {$\scriptstyle $} (y);
\draw[->] (a) to node[swap] {$\scriptstyle $} (b);
\draw[->] (b) to node[swap] {$\scriptstyle $} (y);
\end{tikzpicture}. \]
By Proposition \ref{pro6.1} and Theorem \ref{thm6.7}, the above pushout square can be identified with the following pushout square:
\[\begin{tikzpicture}[auto]
\node (a) at (0, 2) {$F\mathbb L_A\otimes_A^LS$}; \node (x) at (4, 2) {$F\mathbb L_R\otimes_R^LS$};
\node (b) at (0, 0) {$S/p$};   \node (y) at (4, 0) {$0$};
\draw[->] (a) to node {$\scriptstyle $} (x);
\draw[->] (x) to node {$\scriptstyle $} (y);
\draw[->] (a) to node[swap] {$\scriptstyle $} (b);
\draw[->] (b) to node[swap] {$\scriptstyle $} (y);
\end{tikzpicture}. \]
Thus we have \[F\mathbb L_A\otimes_A^LS\simeq (S/p)\oplus (F\mathbb L_R\otimes_R^LS).\] Since $R\to S$ is faithfully flat, it reduces to show that the complex $F\mathbb L_A$ is a free $A/p$-module of rank $\operatorname{dim}A+\operatorname{log}_p[k:k^p]$. By Proposition \ref{pro5.11.1}, there is a canonical equivalence\[(F\mathbb L_{C\llbracket x_1,\dots, x_r\rrbracket})_{(x_1,\dots,x_r)}^{\wedge}\simeq (F\mathbb L_{C[x_1,\dots,x_r]})_{(x_1,\dots,x_r)}^{\wedge}.\] By Lemma \ref{lem6.15} below, $F\mathbb L_{C\llbracket x_1,\dots, x_r \rrbracket}$ is already complete. Thus the above equivalence implies the following equivalence
\[F\mathbb L_{C\llbracket x_1,\dots,x_r\rrbracket}\simeq (F\mathbb L_{C[x_1,\dots,x_r]})_{(x_1,\dots,x_r)}^{\wedge}.\] Therefore, it reduces to show that the complex $F\mathbb L_{C[x_1,\dots,x_r]}$ is free of rank $\operatorname{dim}(C[x_1,\dots,x_r])+\operatorname{log}_p[k:k^p]$. By Proposition \ref{pro5.11.2}, it suffices to show that $F\mathbb L_C$ is free of rank $\operatorname{dim}C+\operatorname{log}_p[k:k^p]$. Consider the fundamental fiber sequence (Theorem \ref{thm5.5})
\[F\mathbb L_{\mathbb Z_{(p)}}\otimes^L_{\mathbb Z_{(p)}}C\to F\mathbb L_C\to F^*(\mathbb L_{k/\mathbb F_p})\to\] 
associated to $\mathbb Z_{(p)}\to C$. Then the left term is a free $k$-module of rank one and the right term is a free $k$-module of rank $\operatorname{log}_p[k:k^p]$, so we get the conclusion.

$(2)\Rightarrow (1)$. This follows form \cite[Theorem 4.5]{Tak26}.
\end{proof}

Next, we remove the completeness assumption in the above Theorem . For this purpose, we first establish the following lemma, which was also used in the above proof.
\begin{lemma}\label{lem6.15}
    Let $R$ be a $p$-torsionfree noetherian ring such that $R/p$ is $F$-finite. Then $\pi_i(F\mathbb L_R)$ is a finitely generated $R/p$-module for any $i\ge0$.
\end{lemma}
\begin{proof}
    Consider the fundamental fiber sequence (Theorem \ref{thm5.5})
    \[F\mathbb L_{\mathbb Z_{(p)}}\otimes_{\mathbb Z_{(p)}}R\to F\mathbb L_R\to F^*(\mathbb L_{(R/p)/\mathbb F_p})\to \] associated to $\mathbb Z_{(p)}\to R$. By this fiber sequence, it reduces to show that $\pi_i(F^*(\mathbb L_{(R/p)/\mathbb F_p}))$ is finitely generated for any $i\ge0$. For simplicity of notation, we write $S:=R/p$. Consider the transitive fiber sequence
    \[\mathbb F^*(L_{S/\mathbb F_p})\to \mathbb L_{F_*S/\mathbb F_p}\to \mathbb L_{F_*S/S}\to\]
    associated to $\mathbb F_p\to S\xrightarrow[]{F}F_*S$. The first map in the above fiber sequence is induced by the Frobenius map, so it is a zero map. Since $S\xrightarrow[]{F} F_*S$ is finite, $\pi_i(\mathbb L_{F_*S/S})$ is finitely generated for any $i\ge0$. Since $S$ is noetherian, these observations imply the conclusion.
\end{proof}

In the first version of this paper, the author could not remove the assumption of completeness in Theorem \ref{thm6.10}. The proof of the following generalized assertion is inspired by R. Takeuchi's talk.
\begin{theorem}\label{thm6.16}
    Let $R$ be a $p$-torsionfree noetherian regular local ring whose residue characteristic is $p$. Assume that $R/p$ is $F$-finite. Then the Frobenius--Witt cotangent complex $F\mathbb L_R$ is concentrated in degree zero.
\end{theorem}
\begin{proof}
    By Proposition \ref{pro5.11.1}, there is a canonical equivalence
    \[(F\mathbb L_R)_{\mathfrak m}^{\wedge}\simeq (F\mathbb L_{R_{\mathfrak m}^{\wedge}})_{\mathfrak m}^{\wedge},\] where $\mathfrak m$ is the maximal ideal of $R$. Since $R$ is noetherian, Corollary \ref{cor6.12} implies that the right hand side term is concentrated in degree zero. Thus it suffices to show that there is a canonical equivalence\[(F\mathbb L_R)_{\mathfrak m}^{\wedge}\simeq F\mathbb L_R\otimes_R^LR_{\mathfrak m}^{\wedge},\] because the canonical map $R\to R_{\mathfrak m}^{\wedge}$ is faithfully flat. To prove this, it suffices to show that $\pi_i(F\mathbb L_R)$ is finitely generated for any $i\ge 0$ by \cite[Tag 0A06]{Sta26}. This is Lemma \ref{lem6.15}.
\end{proof}
The following corollary is a generalization of Corollary \ref{cor6.12}.
\begin{coro}\label{cor6.17}
    Let $R$ be a $p$-torsionfree noetherian local ring with residue field $k$ of characteristic $p$. Assume that $R/p$ is $F$-finite. Then the following are equivalent:\begin{enumerate}
        \item $R$ is regular
        \item $F\mathbb L_R$ is a free $R/p$-module of rank $\operatorname{dim}R+\operatorname{log}_p[k:k^p]$.
    \end{enumerate}
\end{coro}
\begin{proof}
    $(1)\Rightarrow (2)$. We assume that $R$ is regular. By the proof of Theorem \ref{thm6.16}, there is a canonical equivalence:
    \[F\mathbb L_{R_{\mathfrak m}^{\wedge}}\simeq F\mathbb L_R\otimes_R^LR_{\mathfrak m}^{\wedge}.\]
    Since $R\to R_{\mathfrak m}^{\wedge}$ is faithfully flat, the conclusion follows from Corollary \ref{cor6.12}.

    $(2)\Rightarrow (1)$. This follows from \cite[Theorem 4.5]{Tak26}.
\end{proof}

\section{An application to the deformation theory of $\delta$-structures}
In this section, assuming the existence of a Frobenius lift, we describe the obstruction to the existence of a compatible $\delta$-structure, as well as the space of such $\delta$-structures, in terms of the Frobenius--Witt cotangent complex. The question of whether a Frobenius lift arises from a $\delta$-structure was studied by Rezk in \cite{Rez19}, where the following definition was introduced.
\begin{definition}[\cite{Rez19}]
Let $A$ be a ring. The ring $\overline{V}(A)$ is the subring of $R\times R$ defined by \[\overline V(A):=\{(a,b)\in A\times A|\, b\equiv a^p\mod p\}.\]
\end{definition}

\begin{lemma}[\cite{Rez19}]\label{D.1}
    Let $A$ be a ring. Then the following is a square zero extension of rings:
    \[0\to A[p]\xrightarrow[]{a\mapsto (0,a)} W_2(A)\xrightarrow[]{(a,b)\mapsto (a,a^p+pb)}\overline V(A)\to 0.\]
    Furthermore, the induced $\overline V(A)$-module structure on $A[p]$ can be written as $(a,b)c=a^pc$ for any $(a,b)\in \overline{V}(A)$ and $c\in A[p]$.
\end{lemma}
\begin{proof}
    The proof is straightforward and is omitted. See \cite[Proposition 3.2]{Rez19} for details.
\end{proof}
In order to describe the obstruction to the existence of a $\delta$-structure, we introduce the notion of extension groups of pointed modules in the special case needed for our applications.

\begin{definition}
    Let $R$ be an animated ring. Let $(M,m):=(m:R\to M)$ be an object of the coslice category $\mathcal D(R)_{R/}$. Let $N$ be an object of $\mathcal{D}(R)$. For any integer $i$, we define the $i$-th extension group $\operatorname{Ext}^i_R((M,m),(N,0))\in \operatorname{Mod}(\mathbb Z)$ by \[\pi_0(\operatorname{Hom}_{\operatorname{Fun}(\Delta^1,\mathcal D(R))}((M,m),(N,0)[i])),\] where we regard $(M,m)$ (resp. $(N,0)$) as an object of the stable $\infty$-category $\operatorname{Fun}(\Delta^1,\mathcal{D}(R))$ via $m:R\to M$ (resp. $0\to N$).
\end{definition}

The following theorem is the main result of this section.
\begin{theorem}\label{D.2}
Let $A$ be a $\mathbb Z_{(p)}$-algebra equipped with a Frobenius lift $\phi$.
\begin{enumerate}
\item Suppose that $A$ admits a $\delta$-structure $\delta_0$ satisfying $\phi = (-)^p + p\delta_0$.
Then the set
\[\{\delta\in \operatorname{End}_{\operatorname{Set}}(A)|\,\delta \text{ is a }\delta\text{-structure on }A \text{ satisfying} \ \phi=(-)^p+p\delta\}\]
is a torsor under $\operatorname{Ext}^0_{A/^Lp}((F\mathbb L_A,w(p)),(A[p],0))$.

\item There exists an element
\[\xi\in \operatorname{Ext}^1_{A/^Lp}((F\mathbb L_A,w(p)),(A[p],0))\]
such that $\xi$ is zero if and only if $A$ admits a $\delta$-structure $\delta$ satisfying $\phi=(-)^p+p\delta$. Furthermore, the element $\xi$ only depends on $A$ and $\phi$.
\end{enumerate}
\end{theorem}
\begin{proof}
    (1). First, we have \begin{align*}
        \operatorname{Ext}^0_{A/^Lp}((F\mathbb L_A,w(p)),(A[p],0))&\cong \operatorname{Hom}_{\mathcal{D}(A/^Lp)}(\operatorname{Cofib}(A/^Lp\xrightarrow[]{w(p)}F\mathbb L_A),A[p])\\&\cong \operatorname{Hom}_{\operatorname{Mod}(A/p)_{(A/p)/}}((F\Omega_A,w(p)),(A[p],0))\\&\cong \{d\in\operatorname{FWDer}(A,A[p])|\, d(p)=0 \}.
    \end{align*}
    For each Frobenius--Witt derivation $d:A\to A[p]$ with $d(p)=0$ and $\delta$-structure on $A$ with $\phi=(-)^p+p\delta$, the sum $\delta_d:=\delta+d$ is a $\delta$-structure on $A$ satisfying $\phi=(-)^p+p\delta_d$. Indeed, this follows from the following direct computations:
    \begin{align*}
        \delta_d(ab)&=\delta(ab)+d(ab)\\&=a^p(\delta(b)+d(b))+b^p(\delta(a)+d(a))+p\delta(a)\delta(b)\\&\overset{(i)}{=}a^p(\delta(b)+d(b))+b^p(\delta(a)+d(a))+p(\delta(a)+d(a))(\delta(b)+d(b))\\&=a^p\delta_d(b)+b^p\delta_d(a)+p\delta_d(a)\delta_d(b)\\
    \delta_d(a+b)&=\delta(a+b)+d(a+b)\\&=\delta(a)+\delta(b)-P(a,b)+d(a)+d(b)-P(a,b)d(p)\\&\overset{(ii)}{=}\delta(a)+\delta(b)-P(a,b)+d(a)+d(b)\\
        &=\delta_d(a)+\delta_d(b)-P(a,b),
    \end{align*}
    for any $a,b\in A$, where we used $pd(a)=0=pd(b)$ in $(i)$ and $d(p)=0$ in $(ii)$. Therefore, we get a free action of the group $ \operatorname{Ext}^0_{A/^Lp}((F\mathbb L_A,w(p)),(A[p],0))$ on the set of $\delta$-structures $\delta$ on $A$ satisfying $\phi=(-)^p+p\delta$. It remains to show that this action is transitive. Let $\delta$ and $\delta'$ be two $\delta$-structures on $A$ satisfying $\phi=(-)^p+p\delta$ and $\phi=(-)^p+p\delta'$. Write $d:=\delta-\delta'$. Then we can check that this map is a Frobenius--Witt derivation $d:A\to A[p]$ satisfying $d(p)=0$. Indeed, this follows from the following direct computations:
    \begin{align*}
        d(a+b)&=\delta(a+b)-\delta'(a+b)\\&=(\delta(a)+\delta(b))-(\delta'(a)+\delta'(b))\\&=d(a)+d(b),\\
        d(ab)&=(a^p\delta(b)+b^p\delta(a))-(a^p\delta'(b)+b^p\delta'(a))+p\delta(a)\delta(b)-p\delta'(a)\delta'(b)\\&\overset{(iii)}{=}a^p(\delta-\delta')(b)+b^p(\delta-\delta')(a)
    \end{align*}
    for any $a,b\in A$, where we used $p\delta(a)=p\delta'(a)$ and $p\delta(b)=p\delta'(b)$ in $(iii)$.

    (2).Using $\phi$, we can define a ring map $f:A\to \overline V(A):a\mapsto (a,\phi(a))$. $A$ has a $\delta$-structure $\delta$ satisfying $\phi=(-)^p+p\delta$ if and only if this map lifts along the square-zero extension $W_2(A)\twoheadrightarrow \overline V(A)$ of Lemma \ref{D.1}. This is equivalent to the assertion that the pull back of the exact sequence in Lemma \ref{D.1} along $f:A\to \overline V(A)$ splits. Combining these arguments, we see that the obstruction class of the existence of a $\delta$-structure on $A$ satisfying $\phi=(-)^p+p\delta$ is the image of the element of $\operatorname{Ext}^1_{\overline{V}(A)}(\mathbb L_{\overline V(A)/\mathbb Z},A[p])$ corresponding the extension in Lemma \ref{D.1} under the map \[\operatorname{Ext}^1_{\overline{V}(A)}(\mathbb L_{\overline V(A)/\mathbb Z},A[p])\to \operatorname{Ext}^1_A(\mathbb L_{A/\mathbb Z},F_*A[p])\] induced by the map $f:A\to \overline{V}(A)$, where we used the second part of Lemma \ref{D.1} to ensure that the $A$-module $f_*A[p]$ can be written by $F_*A[p]$ using the Frobenius map $F$ on $A/p$. Thus, it remains to show that there is a canonical isomorphism\[\operatorname{Ext}^1_A(\mathbb L_{A/\mathbb Z},F_*A[p])\cong \operatorname{Ext}^1_{A/^Lp}((F\mathbb L_A,w(p)),(A[p],0)).\] This follows from the following computation:\begin{align*}
        \operatorname{Ext}^1_{A/^Lp}((F\mathbb L_A,w(p)),(A[p],0))&=\pi_0(\operatorname{Hom}_{\mathcal D(A/^Lp)}(\operatorname{Cofib}(A/^Lp\xrightarrow[]{w(p)}F\mathbb L_A),A[p][1]))\\&\overset{(i)}{\cong} \pi_0(\operatorname{Hom}_{\mathcal D(A/^Lp)}(F^*(\mathbb L_{A/\mathbb Z}\otimes^L_{\mathbb Z}\mathbb F_p),A[p][1]))\\&\cong \pi_0(\operatorname{Hom}_{\mathcal D(A/^Lp)}(\mathbb L_{A/\mathbb Z}\otimes^L_{\mathbb Z}\mathbb F_p,F_*(A[p][1])))\\&\cong \pi_0(\operatorname{Hom}_{\mathcal D(A)}(\mathbb L_{A/\mathbb Z},F_*(A[p][1])))\\&\cong \operatorname{Ext}^1_A(\mathbb L_{A/\mathbb Z},F_*A[p]),
    \end{align*} where we used Theorem \ref{thm5.5} in $(i)$.
\end{proof}

\section{A conjecture on the cotangent complex over the field of one element}
As we have already seen, the Frobenius--Witt cotangent complex is defined as a complex over $\mathbb F_p$. In this section, assuming the existence of the field with one element, we propose a conjectural extension of the Frobenius--Witt cotangent complex to a complex over $\mathbb Z_{(p)}$.

Let $R\to S$ be a map of $\mathbb Z_{(p)}$-rings. The fundamental fiber sequence (Theorem \ref{thm5.5}) \[F\mathbb L_R\otimes_R^LS\to F\mathbb L_S\to F^*(\mathbb L_{S/R}\otimes_R^LR/^Lp)\to \] can be seen as the Frobenius pullback of the modulo $p$ of a conjectural transitive fiber sequence associated to a conjectural ring maps $\mathbb F_1\to R\to S$:\[\mathbb L_{R/\mathbb F_1}\otimes_R^LS\to \mathbb L_{S/{\mathbb F_1}}\to \mathbb L_{S/R}\to .\] By this observation, it is natural to think the following conjecture:
\begin{conj}
    There exists a functor\[\mathbb L_{(-)/\mathbb F_1}:\operatorname{Ani(Ring)}_{\mathbb Z_{(p)}/}\to \mathcal{D}(\mathbb Z_{(p)})\] of $\infty$-categories such that there exists a natural equivalence: \[F^*(\mathbb L_{R/\mathbb F_1}\otimes_R^LR/^Lp)\simeq F\mathbb L_{R}\] for any animated $\mathbb Z_{(p)}$-algebra $R$, where $F$ is the Frobenius endomorphism on $R/^Lp$.
\end{conj}
\begin{remark}
    This conjecture is reminiscent of a guess proposed by B. Bhatt at the end of \cite{Bha21}.
\end{remark}


\begin{thebibliography}{BMS18}

\bibitem[Bha21]{Bha21}
B.~Bhatt.
\newblock{\em p-adic cohomology theories via stacks}.
\newblock A conference on the occasion of Takeshi Saito's 60th birthday (2021), {\url{https://youtu.be/v2Jfk-NTjp4?si=5qSYokoVrr_4RhJO}}.

\bibitem[BIM19]{BIM19}
B.~Bhatt, S.~B.~Iyengar and L.~Ma.
\newblock{\em Regular rings and perfect(oid) algebras}.
\newblock Comm. Algebra {\bf 47} (2019), no.~6, 2367--2383; MR3957103


\bibitem[BL22]{BL22}
B.~Bhatt, J.~Lurie.
\newblock{\em The prismatization of p-adic formal schemes}.
\newblock arXiv:2201.06124. (2022)

\bibitem[BMS18]{BMS18}
B.~Bhatt, M.~Morrow, P.~Scholze.
\newblock{\em Integral p-adic Hodge theory}.
\newblock Publ. Math. Inst. Hautes Études Sci. 128 (2018), 219--397. MR 3905467
\bibitem [BMS19]{BMS19}
B.~Bhatt, M.~Morrow, P.~Scholze.
\newblock{\em Topological Hochschild homology and integral p-adic Hodge theory}.
\newblock Publ. Math. Inst. Hautes Études Sci. 129 (2019), 199--310. MR3949030

\bibitem[BS22]{BS22}
B.~Bhatt, P.~Scholze.
\newblock{\em Prisms and prismatic cohomology}.
\newblock
Ann. of Math. (2) 196 (2022), no. 3, 1135--1275. MR4502597

\bibitem[ČS24]{CS24}
K.~Česnavičius, P.~Scholze.
\newblock{\em Purity for flat cohomology}.
\newblock Ann. of Math. (2) 199 (2024), no. 1, 51--180. MR4681144

\bibitem[IN26]{IN26}
R.~Ishizuka, K.~Nakazato.
\newblock{\em Prismatic Kunz's theorem}.
\newblock J. Algebra {\bf 693} (2026), 732--769; MR5022261

\bibitem[Iye07]{Iye07}
S.~B. Iyengar.
\newblock{\em Andr\'e-Quillen homology of commutative algebras}.
\newblock {\it Interactions between homotopy theory and algebra} (2007), 203--234, Contemp. Math., 436, Amer. Math. Soc., Providence, RI, ; MR2355775

%\bibitem[Ill71]{Ill71}
%L.~Illusie
%\newblock{\em Complexe cotangent et déformations. I}.
%\newblock 
%Lecture Notes in Math., Vol. 239
%Springer-Verlag, Berlin-New York, 1971. xv+355 pp.
\bibitem[Lur09]{Lur09}
J.~Lurie.
\newblock{\em Higher topos theory}.
\newblock 
Ann. of Math. Stud., 170
Princeton University Press, Princeton, NJ, 2009. xviii+925 pp.
ISBN:978-0-691-14049-0
ISBN:0-691-14049-9

\bibitem[Lur17]{Lur17}
J.~Lurie.
\newblock{\em Higher algebra}.
\newblock
{\url{https://people.math.harvard.edu/~lurie/papers/HA.pdf}, sep 2017}

\bibitem[Lur18]{Lur18}
J.~Lurie.
\newblock{\em Spectral algebraic geometry}.
\newblock 
{\url{https://www.math.ias.edu/~lurie/papers/SAG-rootfile.pdf}, feb 2018.}

\bibitem[Lur26]{Lur26}
J.~Lurie.
\newblock{\em Kerodon}.
\newblock {\url{https://kerodon.net/}, apr 2026.}




\bibitem[Mao24]{Mao24}
Z.~Mao.
\newblock{\em Revisiting derived crystalline cohomology}.
\newblock 
Bull. Soc. Math. France 152 (2024), no. 4, 659–784. MR4851406

\bibitem[Rez19]{Rez19}
C.~Rezk.
\newblock{\em Etale extensions of $\lambda$-rings}.
\newblock{\url{https://rezk.web.illinois.edu/etale-lambda.pdf}}.

\bibitem[Sai22]{Sai22}
T.~Saito.
\newblock{\em Frobenius--Witt differentials and regularity}.
\newblock Algebra Number Theory 16 (2022), no. 2, 369–391. MR4412577

\bibitem[Sta26]{Sta26}
Stacks Project Authors.
\newblock{\em Stacks Project}.
\newblock{\url{https://stacks.math.columbia.edu/}}.

\bibitem[Tak26]{Tak26}
R.~Takeuchi.
\newblock{\em A criterion for log regularity via log Frobenius-Witt differentials}.
\newblock{arXiv:2604.17394. (2026)}.
\end{thebibliography}
\end{document}